\documentclass[11pt,oneside]{amsart}
\usepackage{amscd,amssymb}

\usepackage{color}
\usepackage{graphicx}
\theoremstyle{plain}
\numberwithin{equation}{section}
\usepackage[latin1]{inputenc}
\DeclareMathAlphabet{\mathpzc}{OT1}{pzc}{m}{it}

\theoremstyle{plain}
\newtheorem*{maintheorem*}{Main Theorem}
\newtheorem*{thm*}{Theorem}
\newtheorem*{thma*}{Theorem A}
\newtheorem*{thmaa*}{Theorem A'}
\newtheorem*{thmb*}{Theorem B}
\newtheorem*{thmo*}{Theorem 1.1}
\newtheorem*{thmc*}{Theorem C}
\newtheorem*{thmd*}{Theorem D}
\newtheorem*{thmf*}{Theorem 4.1}
\newtheorem*{remark*}{Remark}
\newtheorem*{conjecture*}{Conjecture}
\newtheorem*{prop*}{Proposition}
\newtheorem*{lem*}{Basic Lemma}
\newtheorem{thm}{Theorem}[section]
\newtheorem{cor}[thm]{Corollary}
\newtheorem{lem}[thm]{Lemma}
\newtheorem{prop}[thm]{Proposition}

\theoremstyle{definition}

\newtheorem*{proofc*}{Proof of Theorem C}

\newtheorem{terminology}[thm]{Terminology}

\newtheorem{definition}[thm]{Definition}

\newtheorem{remark}[thm]{Remark}
\newtheorem{notation}[thm]{Notation}

\def\bbr{\mathbb{R}}

\def\bbc{\mathbb{C}}

\def\bbn{\mathbb{N}}

\def\PSL{\rm{PSL}}

\def\supp{{\rm{supp}}}
\newcommand{\PS}{\operatorname{PS}}
\newcommand{\BR}{\operatorname{BR}}
\newcommand{\BMS}{\operatorname{BMS}}

\renewcommand{\setminus}{-}
\newcommand{\bh}{\partial(\mathbb{H}^3)}
\newcommand{\be}{\begin{equation}}
\newcommand{\ee}{\end{equation}}

\def\e{\varepsilon}
\def\h{\hspace{1mm}}
\def\hh{\hspace{.5mm}}

\def\G{\Gamma}

\def\ba{\backslash}

\newcommand{\bH}{\mathbb H}\renewcommand{\H}{\mathcal H}

\newcommand{\br}{m^{\BR}}
\newcommand{\LG}{\Lambda(\Gamma)}

\newcommand{\mBR}{m^{\BR}}

\newcommand{\Leb}{\operatorname{Leb}}
\newcommand{\T}{\operatorname{T}}

\newcommand{\z}{\mathbb{Z}}

\newcommand{\n}{\mathbb{N}}
\renewcommand{\c}{\mathbb{C}}
\renewcommand{\br}{\mathbb{R}}\renewcommand{\O}{\mathcal{O}}

\newcommand{\op}{\operatorname}
\newcommand{\te}{\textstyle}

\newcommand{\chN}{\check{N}}
\newcommand{\chn}{\check{n}}

\def\mfld{\mathcal{M}}\def\M{\mathcal{M}}

\begin{document}
\title[Burger-Roblin]{Ergodicity of 
 unipotent flows and Kleinian groups}

\author{Amir Mohammadi}
\address{Department of Mathematics, The University of Texas at Austin, Austin, TX 78750}
\email{amir@math.utexas.edu}
\thanks{Mohammadi was supported in part by NSF Grant \#1200388.}

\author{Hee Oh}
\address{Mathematics department, Yale university, New Haven, CT 06520 
and Korea Institute for Advanced Study, Seoul, Korea}

\email{hee.oh@yale.edu}
\thanks{Oh was supported in part by NSF Grant \#1068094.}



\begin{abstract} 
 Let $\mfld$ be a non-elementary
convex cocompact  hyperbolic $3$-manifold and $\delta$ be the critical exponent of
 its fundamental group.
We prove that a one-dimensional unipotent flow for the frame bundle of $\mfld$ is ergodic for the Burger-Roblin measure
 provided $\delta>1$.
\end{abstract}

\maketitle \tableofcontents
\section{Introduction}
In this paper we study dynamical properties of one-parameter unipotent flow
for the frame bundle of a convex cocompact hyperbolic $3$-manifold $\mfld$. 
When the critical exponent of the
fundamental group $\pi_1(\mfld)$ exceeds one,
we show that this flow is conservative and ergodic
 for the Burger-Roblin measure $m^{\BR}$: almost all points enter to a given Borel subset of positive measure for
 an unbounded amount of time.
Such a manifold admits a unique positive square-integrable eigenfunction $\phi_0$
 of the Laplacian  with base eigenvalue.
Our  result implies that 
a randomly chosen unipotent orbit, normalized by the time average of the eigenfunction $\phi_0$, becomes equidistributed
with respect to the Burger-Roblin measure.

\medskip

To state our result more precisely, let  $G=\PSL_2(\c)$, which is
 the group of orientation preserving isometries of the hyperbolic space $\bH^3$.
Let $\G$ be a non-elementary, torsion-free, discrete subgroup of $G$ which is {\it convex cocompact}, that is, the convex core
of $\G$ is compact. Equivalently, $\G\ba \bH^3$ admits a finite sided fundamental domain with no cusps.
 Convex cocompact groups arise in topology 
as fundamental groups of compact hyperbolic $3$-manifolds with totally geodesic boundary.

The frame bundle of the manifold $\M=\G\ba \bH^3$, which is a circle bundle
over the unit tangent bundle $\T^1(\M)$, is identified with the homogeneous space $X=\G\ba G$.
We consider the unipotent flow on $X$ given by
the right translations of the one-parameter unipotent subgroup 
\begin{equation}\label{u} U=\lbrace u_t:= \begin{pmatrix} 1& 0 \\ t & 1\end{pmatrix}
 : t\in \br\rbrace.\end{equation} 

This flow is called {\it ergodic} with respect to a fixed locally finite Borel measure on $X$,  if
any invariant Borel subset is either null or co-null. We denote by $\delta$
the critical exponent of $\G$, which is equal to
the Hausdorff dimension of the limit set of $\G$ (\cite{Yau}, \cite{Sullivan1979}).
When $\delta=2$, $X$ is compact \cite{Sullivan1979}
and the classical Moore's theorem in 1966 \cite{Mo}
implies that this flow is ergodic with respect to the volume measure, i.e., 
the $G$-invariant measure.
When $\delta<2$, the volume measure is not ergodic any more, and furthermore,
Ratner's measure classification theorem \cite{Ratner2} says that
there exists no {\it finite} $U$-ergodic invariant measure on $X$.
 This raises
a natural question of finding a locally finite $U$-ergodic measure on $X$.
Our main result in this paper is that when $\delta>1$,
the Burger-Roblin measure is conservative and ergodic.

The conservativity means that  for any subset $S$ of positive measure,
the $U$-orbits of almost all points in $S$ spend an infinite amount of time in $S$. 
Any finite invariant measure is conservative by the Poincar\'e recurrence theorem. 
For a general locally finite invariant measure,
 the Hopf decomposition theorem \cite{Kre} says that any
ergodic  measure is either conservative or totally dissipative (i.e., for any Borel subset $S$,
$xu_t \notin S$ for all large  $|t|\gg 1$ and a.e. $x\in S$).
For $\delta<2$, there are many isometric embeddings of the real line in $X$, by $t\mapsto xu_t$, giving
rise to a family of dissipative ergodic measures for $U$.

\medskip

We refer to the Burger-Roblin measure as the BR measure for short, and
give its description using the Iwasawa decomposition $G=KAN$:
$K=\op{PSU}_2$, $A=\{a_s: s\in \br\}$, $N=\{n_z: z\in \c\}$
where $$a_s=\begin{pmatrix} e^{s/2}& 0\\ 0 & e^{-s/2} \end{pmatrix}  \;\; \text{and}\;\;  n_z=
\begin{pmatrix} 1& 0\\ z& 1 \end{pmatrix}.
$$
Furthermore let $M$ denote the centralizer of $A$ in $K.$

The groups $A$ and $N$ play important roles in dynamics as the right translation by $a_s$ on $X$
is the frame flow, which is the extension of the geodesic flow on $\T^1(\M)$
and $N$-orbits give rise to unstable horospherical foliation on $X$ for the frame flow.

Fixing $o\in \bH^3$ stabilized by $K$, we denote by $\nu_o$ the Patterson-Sullivan measure
on the boundary $\partial(\bH^3)$, supported on the limit set of $\G$, associated to $o$ (\cite{Patterson1976}, \cite{Sullivan1979}),
and refer to it as the PS measure.
Sullivan showed that the PS measure coincides with the $\delta$-dimensional
Hausdorff measure of the limit set of $\G$. 
Using the transitive action of $K$ on $\partial(\bH^3)=K/M$, 
we may lift $\nu_o$ to an $M$-invariant measure on $K$.

\medskip
\noindent{\bf Burger-Roblin measure} 
Define the measure $\tilde m^{\BR}$ on $G$ as follows: for $\psi\in C_c(G)$,
$$\tilde m^{\BR}(\psi)=\int_{G} \psi(k  a_sn_z) e^{-\delta s} d\nu_o(k) ds\;dz$$
where $ds$ and $dz$ are the Lebesgue measures on $\br$ and $\c$ respectively.
It is left $\G$-invariant and right $N$-invariant. 
The $\BR$ measure $m^{\BR}$ is  a locally finite measure on $X$ induced by $\tilde m^{\BR}$.
When $\delta=2$, $m^{\BR}$ is simply a $G$-invariant measure, but
it is an infinite measure if $\delta<2$.

Roblin showed that
the BR measure is the unique  $NM$-invariant ergodic measure on $X$  
which is not supported on a closed $NM$-orbit in $X$ \cite{Roblin2003}.
 For $\G$ Zariski dense (which is the case if $\delta>1$),
 Winter \cite{Wi}  proved that
 $m^{\BR}$ is $N$-ergodic, and this implies that $m^{\BR}$ is  the unique  $N$-invariant ergodic measure on $X$  
which is not supported on a closed $N$-orbit in $X$, by Roblin's classification. We note that the analogous result for  $G=\PSL_2(\br)$ was established earlier by Burger \cite{Burger1990}
 when $\G$ is convex-compact with $\delta>1/2$.


The main result of this paper is:
\begin{thm}\label{main} Let $\G$ be a convex cocompact subgroup of $G$
which is not virtually abelian. 
The $U$-flow on $(X, m^{\BR})$  is 
ergodic if $\delta>1.$
\end{thm}


We also show the conservativity of the BR-measure for $\delta>1$,
without knowing its ergodicity a priori.

\begin{remark}\rm 
We remark that most of arguments in the proof of Theorem \ref{main} works for a higher dimensional case as well.
Namely, the same proof will show that if $G$ is the group of orientation preserving isometries of the hyperbolic $n$-space,
$\G$ is a Zariski dense, convex cocompact subgroup 
of $G$,  $U$ is a $k$-dimensional connected unipotent subgroup
of $G$, then the $U$ action is ergodic with respect to the BR-measure on $\Gamma\ba G$ if $\delta>n-k$.
\end{remark}

For a probability measure $\mu$ on $X$, the Birkhoff pointwise ergodic theorem (1931) says that 
the ergodicity of a measure preserving flow $\{u_t\}$ implies that the time average of a typical orbit
converges to the space average: for any $\psi\in L^1(X)$ and a.e. $x\in X$, as $T\to \infty$,
\be\label{b1} \frac{1}{T}{\int_0^T \psi(xu_t) dt}\longrightarrow \int_X \psi \; d\mu .\ee


A generalization of the Birkhoff theorem for an infinite locally finite conservative ergodic measure
was obtained by E. Hopf \cite{Hopf} in 1937 and says that the ratio of time averages
of a typical orbit for two functions
converges to the ratio of the space averages: for any $\psi_1,\psi_2\in L^1(X)$ with $\psi_2\ge 0$ with $\int_X \psi_2\;d\mu>0$,
as $T\to \infty$,
\be\label{h1}\frac{\int_0^T \psi_1(xu_t) dt}{\int_0^T \psi_2(xu_t) dt}
\longrightarrow  \frac{\int_X \psi_1 \; d\mu}{\int_X \psi_2 \; d\mu} 
\quad \text{ a.e. $x\in X$}.\ee

For our $X=\G\ba G$ with $\G$ convex cocompact and $\delta>1$, 
there is a unique  positive eigenfunction $\phi_0\in L^2(\M)$
for the Laplacian with the smallest eigenvalue
 $\delta(2-\delta)$ and with $\|\phi_0\|_2=1,$ \cite{Sullivan1979}.
In the upper half-space coordinates, $\bH^3=\{z +jy: z \in \c, y >0\}$ with $\partial(\bH^3)=\c\cup\{\infty\}$, the lift $\tilde{\phi}_0$ of $\phi_0$ to $\bH^3$
is realized explicitly as the integral of a Poisson kernel against the PS measure $\nu_o$ (with $o=j$):
$$\tilde{\phi}_0(z+jy )= \int_{\xi\in \c}  \left(\frac{(\|\xi\|^2+1) y}{\|z-\xi\|^2+y^2} \right)^\delta \; d\nu_o(\xi).$$

The $\BR$ measure on $X$ projects down to the absolutely continuous 
measure on the manifold $\M$ and its Radon-Nikodym derivative
with respect to the hyperbolic volume measure is given by
$\phi_0$. 

We deduce the following from Theorem \ref{main} and Hopf's ratio theorem \eqref{h1}:
\begin{cor} Let $\delta>1$.
\begin{enumerate}
\item  For $m^{\BR}$ almost all $x\in X$, the projection of $xU$ to $\mfld$ is dense.

\item For any $\psi\in L^1(X, m^{\BR})$ and for almost all $x\in X$,
$$\lim_{T\to \infty}\frac{\int_0^T \psi(xu_t) dt}{ \int_0^T \phi_0(xu_t) dt} = \int_X \psi\; dm^{\BR} .$$
\end{enumerate}
\end{cor}

\medskip

We explain the proof of Theorem \ref{main} in the case $\delta>1$, in comparison with the finite measure case.
This account makes our introduction a bit too lengthy but we hope that this
will give a summary of the main ideas of the proof which will be helpful to
the readers. 
The proof of Moore's ergodicity theorem is 
 based on the following equivalence
 for a {\it finite} invariant measure $\mu$: $\mu$ is ergodic if and only if
any $U$-invariant function of $L^2(X, \mu )$
is constant a.e.
Through this interpretation, his ergodicity theorem follows from a theorem in the unitary representation
theory that 
any $U$-invariant vector in the Hilbert space $L^2(X,\mu_G)$
is $G$-invariant for the volume measure $\mu_G$.

For an infinite invariant measure,
its ergodicity cannot be understood merely via $L^2$-functions, but we must investigate
all invariant bounded measurable functions. This means that we cannot depend on a convenient theorem
on the dual space of $X$, but rather have to work with the geometric properties of flows
in the space $X$ directly.
We remark that as we are working with a
unipotent flow as opposed to a hyperbolic flow,
 the Hopf argument using the stable and unstable foliations of flows,
 which is a standard tool in studying the ergodicity for hyperbolic flows, is irrelevant here.  

We
use the polynomial divergence property of unipotent flows
to establish that almost all $U$-ergodic components of $m^{\BR}$ are
invariant under the full horospherical subgroup $N$.
The $N$-ergodicity of the BR measure then implies
the $U$-ergodicity as well. This approach has been noted by Margulis as an alternative approach
to show the ergodicity of the volume measure $\mu_G$ in the finite volume case.

However, carrying out this argument in an {\it infinite} measure case
is subtler.
Indeed the heart of the argument, as is explained below, lies in the study of two
nearby orbits in the ``intermediate range''. To the best of our knowledge,
such questions in infinite measure spaces have not been understood before.
 
Let us present a sketch of the argument in the probability measure case.
Let $(X, \mu)$ be a probability measure space.
Then it is straightforward from \eqref{b1} that 
for any generic point $x$, any $0<r<1$,  and any $\psi\in C_c(X)$ 
\be\label{w1} \frac{1}{(1-r)T} {\int_{rT}^T \psi (xu_t) dt}\to \int_X \psi(x) d\mu .\ee
Statements of this nature will be called a ``window theorem" in the sequel. 

We now explain how a suitable window theorem can be used in acquiring an additional invariance
by an element of $N-U.$ This idea was used by Ratner; see~\cite{Ratner1, Ratner2} and the
references therein. We also refer to~\cite{Mar1, Mar2} where similar ideas were used by Margulis  in the topological setting.

Let $\chN$ and $\check{U}$ denote the transpose of $N$ and $U$ respectively. 
Denote by
$N_G(U)$ the normalizer of $U$ in $G$.

Choose  sequences of generic points
$x_k$ and $y_k$ inside a suitably chosen compact subset of $X,$ 
moreover suppose that $y_k=x_kg_k$ with $g_k\notin N_G(U)$ and $g_k\to e$.  Put $\check{V}=\left\{
\begin{pmatrix} 1& it\\ 0& 1 \end{pmatrix}:t\in \br \right\}$,
and assume that the $\check{V}$-component and the $\check{U}$-component
\footnote{these components are well defined for all $g_k$ close enough to $e$.} 
of $g_k$ are  of ``comparable'' size. 

Flowing by $u_t$, we compare the orbits $x_ku_t$ and $y_ku_t=x_ku_t(u_t^{-1}g_k u_t)$.
The divergence properties of unipotent flows (a simple computation in our case), in view
of our above assumption on $g_k$'s, 
says that the divergence of the two orbits is comparable to $u_t^{-1}g_k u_t.$
Furthermore, the $(2,1)$-matrix entry of $u_t^{-1}g_k u_t$ dominates other matrix entries. Let $p(t)$ 
denote the $(2,1)$-matrix entry of $u_t^{-1}g_k u_t.$ This is
a polynomial of degree two whose leading coefficient has comparable
real and imaginary parts.  
Therefore, the divergence of the two orbits is ``essentially'' in the direction of $N-U$.
Choose a sequence of times $T_k$ 
so that $p(T_k)$ converges to a non-trivial element $v\in N-U$.
Letting $\e>0$ be small, since $p(t)$ is a polynomial, 
$y_ku_t$ remains within an $O(\e)$-neighborhood
of $x_ku_t v$ for any $t\in [(1-\e)T_k, T_k]$.
Hence the window theorem \eqref{w1} applied to the sequence of windows  $[(1-\e)T_k, T_k]$
implies that $\mu(\psi)-\mu(v. \psi)=O(\e)$ and hence $\mu(\psi)=\mu(v.\psi)$ as $\e>0$ is arbitrary.
Repeating this process for a sequence of $v_n\rightarrow e,$ we obtain that the measure $\mu$ is invariant
under $N.$

We now turn our attention to an infinite measure case, assuming $\delta>1$. There is a subtle difference for
the average over the one-sided interval $[0,T]$ and over the two sided $[-T, T]$, and the average over $[-T, T]$
is supposed to behave more typically in infinite ergodic theory. 
We first prove that the $\BR$ measure $m^{\BR}$ is $U$-conservative based on a theorem of Marstrand \cite{Mars}, which allows
us to write an ergodic decomposition $m^{\BR}=\int_x \mu_x $ where $\mu_x$ is conservative for a.e. $x$.  Letting $x$ be a generic point for Hopf's ratio theorem and $I_T=[-T, T]$,
in order to deduce
$ \textstyle\frac{\int_{I_T\setminus  I_{rT}} \psi_1(xu_t) dt}{\int_{I_T-I_{rT}} \psi_2(xu_t) dt}
\sim \frac{\mu_x(\psi_1)}{\mu_x(\psi_2)},$
it is sufficient to prove that there is some $c>0$ such that for all $T\gg 1$,
\begin{equation}\label{int1}
\int_{I_T -I_{rT} } \psi_2(xu_t) dt \geq  c \int_{I_T} \psi_2(xu_t) dt. 
\end{equation}

This type of inequality requires strong control on the recurrence of the flow, 
and seems unlikely that \eqref{int1} can be achieved for a set of positive measure, see ~\cite[Section 2.4]{Aa}.
Hence formulating a proper replacement of this condition  \eqref{int1} and its proof 
are simultaneously the hardest part  and at the heart of the proof of Theorem \ref{main}.

We call $x\in X$ a BMS point if both the forward and backward endpoints of the geodesic determined by $x$
belong to the limit set of $\G$. These points precisely comprise the support of the Bowen-Margulis-Sullivan measure
$m^{\BMS}$ on $X$, which is the unique  measure of maximal entropy for the geodesic flow, 
up to a multiplicative constant; see Section~\ref{ph}.
We will call $m^{\BMS}$ the
 BMS measure for simplicity. The support of $m^{\BMS}$ is contained in the convex core of $\G$,
 and in particular a compact subset.
By a BMS box, we mean a subset of the form
$x_0\chN_\rho A_\rho N_\rho M $ where $x_0\in X$ is a BMS point, $\rho>0$ is 
 at most the injectivity radius at $x_0$ and $S_\rho$ means the $\rho$-neighborhood
of $e$ in $S$ for any $S\subset G$.
\begin{thm}[Window Theorem]\label{wt} Let $\delta>1$.
 Let $E\subset X$ be a $\BMS$ box and $\psi \in C_c(X)$ be a non-negative function with $\psi|_{E}>0$.
Then there exist $0<r<1$ and $T_0> 1$ such that for any $T\ge T_0$,
\begin{equation*}
 m^{\BR}\{x\in E: \int_{-rT}^{rT} \psi(xu_t) dt \le (1-r) \int_{-T}^{T} \psi(xu_t) dt\}> \tfrac{r}{2}\cdot  m^{\BR}(E) .\end{equation*}
\end{thm}

We call $x$ a {\it good} point for the window $I_T-I_{rT}$ if
\[
\int_{I_{rT}} \psi(xu_t) dt \le (1-r) \int_{I_T} \psi(xu_t)dt,
\]
or equivalently if
$\int_{I_T-I_{rT}} \psi(xu_t) dt \ge  r \int_{I_T} \psi(xu_t)dt$.
The window theorem says that
the set of good points for the window $I_T-I_{rT}$ has a positive proportion of $E$ for all large $T$.
It follows that for any $\e>0$, we can choose a sequence $T_k=T_k(\e)$ such that
the set $E_k,$ of good points for the window $[(1-\e)T_k ,T_k]$ (or $[-T_k, -(1-\e)T_k]$),
has positive measure. Let $x_k, y_k=x_kg_k\in E_k.$ 
To be able to use this in obtaining an additional invariance, we need to control
the size as well as the direction of the divergence $u_{T_k}^{-1} g_ku_{T_k}$. More precisely, we need to be able to choose
our generic points $y_k=x_kg_k$  so that the size of $g_k$ is comparable with $\tfrac{1}{T_k^2}$ and the size of its $\check{V}$-component
is comparable with that of $\check{U}$-component. 

We emphasize here that we work in the opposite order of a standard way
of applying the pointwise ergodic theorem where one is usually given a sequence $g_k$ and then find
window $I_{T_k}$ depending on $g_k$ (as the window theorem works for any $T_k$).
In our situation, we cannot choose ${T_k}$,
and rather have to work with given $T_k$ (depending on $\e$). So only after we know which $T_k$'s give good
windows for $\e$-width, we can choose good points $x_kg_k$ for those windows.
What allows us to carry out this process is that we have a good understanding of the structure of the generic set along contracting leaves.
To be more precise, the $\PS$-measures on the contracting leaves are basically
$\delta$-dimensional Hausdorff measures on $\bbr^2$, and the assumption that $\delta>1$ enables 
us to find $g_k$ for the ``right scale", see Section~\ref{s;PS-non-fuc}.

Hoping to have given some idea about how the above window theorem \ref{wt} will be used,
we now discuss its proof, which is based on the interplay between the BR measure and the BMS measure.
We mention that
the close relationship between the $\BR$ and the $\BMS$ measure
is also the starting point of Roblin's unique ergodicity theorem for $NM$-invariant measures.

 Unlike the finite measure case,
$m^{\BR}$ is not invariant under the frame flow, which is the right translation by $a_s$ in $X$.
However, as $s\to +\infty$, the normalized measure $\mu_s^{\BR}:=(a_{-s})_* m^{\BR}|_E$ (the push-forward 
 of the restriction $m^{\BR}|_E$ by the frame flow
$a_{-s}$) converges to $m^{\BMS}$ in the weak* topology.

Under the assumption $\delta>1,$
the BMS measure turns out to be 
$U$-recurrent and hence almost all of its $U$-leafwise
measures are non-atomic. This will imply
that the analogue of \eqref{int1} holds for ``most" of the 
$U$-leafwise measures of $m^{\BMS}$.

The goal is to utilize this and the fact that $\mu_s^{\BR}$ weakly converges to
$m^{\BMS}$, in order to deduce that many of 
the $U$-leafwise measures of $m^{\BR}$ 
must also satisfy  \eqref{int1}.  We mention that
in general it is rather rare to be able to deduce ``interesting" statements
regarding leafwise measures from weak* convergence
of measures. One possible explanation for this is that the leafwise measures 
of a sequence of measures may change ``very irregularly" as one moves 
in the transversal direction, e.g. approximation of Lebesgue measure by atomic measures.        

We succeed here  essentially because we have a rather 
good understanding of the $N$-leafwise measures of $\mu_s^{\BR}.$
To be more precise, we can show (i) the $N$-leafwise measures 
of $\mu_s^{\BR}$ change rather regularly, see Section~\ref{sec;cond}, 
furthermore, (ii) the projection of an $N$-leafwise measure of $\mu_s^{\BR}$ 
converges in the $L^2$-sense to its counterpart of $m^{\BMS}$ in most directions,
see Section~\ref{sec;projection}.

We emphasize that we establish the $L^2$-convergence of these measures,
 not merely the weak* convergence, and this is crucial to our proof; 
see the Key Lemma~\ref{zerotwo} and Section~\ref{sec;window}.  
The proof of this $L^2$-convergence requires a certain control
of the energy of the conditional measures of $\mu_s^{\BR}$ which is uniform for all $s\gg 1$.
Our energy estimate is obtained
using the following deep property of the PS measure: for all  $\xi$ in the limit set of $\G$ and for all small $r>0$,
 $\nu_o(B(\xi,r)) \asymp r^\delta$ (with the implied constant being independent of $\xi$ and $r$),
together with the Besicovitch covering lemma. 
Lastly we remark that our proof of the window theorem makes use 
of the rich theory of entropy and is inspired by the low entropy 
method developed by Lindenstrauss in \cite{L}.

\medskip

\noindent{\bf Acknowledgment}
We are very grateful to Tim Austin for numerous helpful discussions regarding various aspects
 of this project. We also thank Chris Bishop and Edward Taylor for helpful correspondences 
regarding totally disconnected limit sets of Kleinian groups.

\section{Ergodic properties of BMS and BR measures}
\subsection{Measures on $\T^1(\G\ba \bH^3)$ associated to a pair of conformal densities}
 Let $(\bH^3, d)$ denote the hyperbolic $3$-space and
 $\partial(\bH^3)$
its geometric boundary.  We denote by $\op{T}^1(\bH^3)$ the unit
tangent bundle of $\bH^3$ and by $\pi$ the natural projection from
$\op{T}^1(\bH^3)\to \bH^3$.

Denote by
 $\{g^s: s\in \br\}$ the geodesic flow. For $u\in \op{T}^1(\bH^3)$,
 we set
 $$u^+:=\lim_{s\to\infty}g^s(u)\quad\text{and}\quad
 u^-:=\lim_{s\to -\infty}g^s(u)$$ which are respectively the forward and backward endpoints
 in $\partial(\bH^3)$ of the geodesic defined by $u$.

\begin{definition}{\rm \begin{enumerate}\item
The Busemann function $\beta:\bh\times \bH^3\times \bH^3\to \br$
 is defined as follows: for $\xi\in \bh$ and $x, y\in \bH^3$,
$$\beta_\xi(x,y)=\lim_{s\to \infty}d(x,\xi_s)-d(y,\xi_s)$$
where $\xi_s$ is a geodesic ray tending to $\xi$ as $s\to \infty$
from a base point $o\in\bH^3,$ fixed once and for all.

 \item  For $u\in \op{T}^1(\bH^3)$,
 the unstable horosphere ${\H}^+_u$ and the stable horosphere $\check{\H}_{u}$ denote respectively the subsets
 $$\{v\in \op{T}^1(\bH^3): v^-=u^-, \beta_{u^-}(\pi(u),\pi(v))=0\} ;$$
$$ \{v\in \op{T}^1(\bH^3): v^+=u^+, \beta_{u^+}(\pi(u),\pi(v))=0\}.$$
 \end{enumerate} }\end{definition}

Each element of the group $\PSL_2(\c)$ acts on $\hat \c=\c\cup\{\infty\}$ as a Mobius transformation and
its action extends to an isometry of $\bH^3$, giving the identification of $\PSL_2(\c)$ as 
 the group of orientation preserving isometries of $\bH^3$. Note that
$(g(u))^\pm=g(u^\pm)$ for $g\in G$.
 The map $ \op{T}^1(\bH^3)\to \partial(\bH^3)$ given by $u\mapsto u^+$
 is called the {\it visual} map.

For discussions in this section, we refer to \cite{Roblin2003}, \cite{OS} and \cite{MO}.
Let $\G$ be a non-elementary (i.e., non virtually abelian) torsion-free discrete subgroup of $G$. 
 Let
$\{\mu_x:x\in \bH^3\}$ be a {\em $\G$-invariant conformal
density\/}
 of dimension $\delta_\mu > 0$  on $\partial(\bH^3)$. That is, each
$\mu_x$ is a non-zero finite Borel measure on $\partial(\bH^3)$
satisfying for any $x,y\in \bH^3$, $\xi\in \partial(\bH^3)$ and
$\gamma\in \G$,
$$\gamma_*\mu_x=\mu_{\gamma x}\quad\text{and}\quad
 \frac{d\mu_y}{d\mu_x}(\xi)=e^{-\delta_\mu \beta_{\xi} (y,x)}, $$
where $\gamma_*\mu_x(F)=\mu_x(\gamma^{-1}(F))$ for any Borel
subset $F$ of $\partial(\bH^3)$. 

Let $\{\mu_x\}$ and $\{\mu_x'\}$ be $\G$-invariant conformal
densities on $\partial(\bH^3)$ of dimension $\delta_\mu$ and
$\delta_{\mu'}$ respectively.
 Following Roblin \cite{Roblin2003}, we define a measure $m^{\mu,\mu'}$ on $\T^1(\Gamma\ba \bH^3)$
associated to the pair $\{\mu_x\}$ and $\{\mu_x'\}$. Note that, fixing $o\in
\bH^3$, the map
$$u \mapsto (u^+, u^-, \beta_{u^-} (o,\pi(u)))$$
is a homeomorphism between $\op{T}^1(\bH^3)$ with
 $(\partial(\bH^3)\times \partial(\bH^3) - \{(\xi,\xi):\xi\in \partial(\bH^3)\})  \times \br .$
\begin{definition} \label{defbmsr}
\rm Set
$$d \tilde m^{\mu,\mu'}(u)=e^{\delta_\mu \beta_{u^+}(o, \pi(u))}\;
 e^{\delta_{\mu'} \beta_{u^-}(o,\pi(u)) }\;
d\mu_o(u^+) d\mu'_o(u^-) dt .$$
It follows from the $\G$-conformal properties of $\{\mu_x\}$ and $\{\mu_x'\}$ that 
$\tilde m^{\mu,\mu'}$ is $\G$-invariant
and that this definition is independent of the choice of $o\in \bH^3$. Therefore it induces a
locally finite Borel measure  $m^{\mu,\mu'}$
 on $\op{T}^1(\G\ba \bH^3)$. 
\end{definition}

\subsection{BMS and BR measures on $\T^1(\G\ba \bH^3)$}
Two important densities we will consider are  the
Patterson-Sullivan density and the $G$-invariant density.

 We
denote by $\delta$ the critical exponent of $\G$, that is, the
abscissa of convergence of the Poincare series
    $\mathcal P_\G(s):=\sum_{\gamma\in \G} e^{-sd(o, \gamma (o))}$ for $o\in
    \bH^3$. As $\G$ is non-elementary, we have $\delta>0$.
The limit set $\Lambda(\G)$ is the set of all accumulation points of orbits $\G(z)$, $z\in \bH^3$.
 As $\G$ acts properly discontinuously on $\bH^3$, $\Lambda(\G)\subset \partial(\bH^3)$.
 Generalizing the
work of Patterson \cite{Patterson1976} for $n=2$, Sullivan
\cite{Sullivan1979} constructed a $\G$-invariant conformal density
$\{\nu_x: x\in \bH^3\}$ of dimension $\delta$  supported on
$\Lambda(\G)$.
Fixing $o\in \bH^3$, each $\nu_x$ is  the unique weak limit as $s\to
\delta^+$ of the family of measures on the compact space $\overline \bH^3:=\bH^3\cup\partial_\infty(\bH^3)$:
\begin{equation*}
\nu_{x,s}:=\frac{1}{\sum_{\gamma\in \G} e^{-sd(o, \gamma (o))}}\sum_{\gamma\in \G} e^{-sd(x, \gamma (o))} \delta_{\gamma
(o)} \end{equation*}
where $\delta_{\gamma(o)}$ is the dirac measure at $\gamma(o)$.
 This family will be referred to as the PS density.
When $\G$ is of divergence type, i.e., 
$\mathcal P_\G(\delta)=\infty$, the PS-density is the unique $\G$-invariant conformal density of
dimension $\delta$ (up to a constant multiple) and atom-free \cite[Cor. 1.8]{Roblin2003}.

We denote by $\{m_x:x\in \bH^3\}$ a $G$-invariant conformal
density on the boundary $\partial(\bH^3)$ of dimension $2$, unique
up to homothety. In particular, each $m_x$ is invariant under the
maximal compact subgroup which stabilizes $x$.
\begin{definition}\label{defbms}\label{brm}
\rm {\begin{enumerate}
\item The measure $m^{\nu,\nu}$ on $\op{T}^1(\G\ba \bH^3)$ is
    called the Bowen-Margulis-Sullivan measure $m^{\BMS}$
    associated with $\{\nu_x\}$  \cite{Sullivan1984}:
 $$m^{\BMS}(u)= e^{\delta \beta_{u^+}(o,
\pi(u))}\;
 e^{\delta \beta_{u^-}(o, \pi(u)) }\; d\nu_o(u^+) d\nu_o(u^-) dt. $$
\item The measure $m^{\nu, m}$ on $\op{T}^1(\G\ba \bH^3)$  is called the Burger-Roblin measure
    $m^{\BR}$ associated with $\{\nu_x\}$ and $\{m_x\}$
    (\cite{Burger1990}, \cite{Roblin2003}):
 $$m^{\BR}(u)= e^{2 \beta_{u^+}(o,
\pi(u))}\;
 e^{\delta \beta_{u^-}(o, \pi(u)) }\; dm_o(u^+) d\nu_o(u^-) dt. $$
\end{enumerate} }
\end{definition}

We will refer to these measures as the BMS and the BR measures respectively for short.
It is worth mentioning that the Riemannian volume measure, in these coordinates, is $m^{m,m}.$

The quotient $\G\ba C(\Lambda(\Gamma))$ 
of the convex hull $C(\Lambda(\G))$ of the limit set modulo $\G$ is called the convex core of $\Gamma$, denoted by $C(\G)$.
A discrete subgroup $\G$ of $G$ is called {\it geometrically finite} if
a unit neighborhood of the convex core $C(\G)$  has finite volume. It is
equivalent to saying that $\G\ba \bH^3$ admits a finite sided fundamental domain.
A geometrically finite group $\G$ is called {\it convex cocompact} if one of the following three equivalent conditions hold
(cf. \cite{Bowditch1993}):
\begin{enumerate}
 \item $C(\Gamma)$ is compact;
\item  $\G\ba \bH^3$ admits a finite sided fundamental domain with no cusps;  
\item  $\Lambda(\G)$ consists only of radial limit points: $\xi\in \Lambda(\G)$ is radial
if any geodesic ray $\xi_t$ toward $\xi$ returns to a compact subset for an unbounded sequence of $t$. 
\end{enumerate}

The BMS measure is invariant under the geodesic flow. Sullivan showed that for $\G$ geometrically finite,
 it is ergodic and moreover the unique measure of maximal entropy (\cite{Sullivan1984},
\cite{OP}). For $\G$ convex cocompact,
the support of the BMS measure is compact, as its projection is contained in $C(\G)$.

\begin{thm}\cite{FS}\label{zd}
If $\G$ is geometrically finite and Zariski dense, the $\PS$ density of any proper Zariski
subvariety of $\partial (\bH^3)$ is zero. 
\end{thm}

\subsection{BMS and BR measures on $X=\G\ba G$}\label{ph}
We fix a point $o\in \bH^3$ whose stabilizer group is
$K:=\op{PSU}(2)$. Then the map $g\mapsto g(o)$ induces a
$G$-equivariant isometry between $G/K$ and $ \bH^3$.
Set
$$M:= \{m_\theta=\text{diag}(e^{i\theta}, e^{-i\theta})\} .$$
 By choosing the unit tangent vector $X_0$ based at $o$ stabilized by $M$,
 $G/M$ can be identified with the unit tangent bundle $\T^1(\bH^3)$ via the orbit map $g\mapsto g(X_0)$.
 This identification can also be lifted to the
identification of the frame bundle of $\bH^3$ with $G$.
These identifications are all $\G$-equivariant and induce identifications of
 the frame bundle of the manifold $\G\ba \bH^3$ with $\G\ba
G$. We set $X=\G\ba G$.
Abusing the notation, we will denote by $m^{\BMS}$ and $m^{\BR},$ respectively, 
the $M$-invariant lifts
of the BMS and the BR measures to $X$.
For $g\in G$, we set $g^{\pm}=(gM)^{\pm}$ where $gM\in
G/M=\T^1(\bH^3)$. 

For $x=\G \ba \G g$, we write $x^\pm\in \Lambda(\G)$ if 
$g^\pm\in \Lambda(\G)$; this is well-defined independent of the choice of $g.$
With this notation, the supports of $m^{\BMS}$ and
$m^{\BR}$ are given respectively by 
$$\Omega:=\{x\in X: x^+,
x^-\in \Lambda(\G)\}\;$$
and $$\Omega_{\BR}:=\{x\in X: x^-\in
\LG\} .$$
The right translation action of the diagonal subgroup
$$A:=\{a_s:=\text{diag}(e^{s/2}, e^{-s/2}): s\in \br\}$$
on $G$ is called the frame flow and it projection to $G/M$
corresponds to the geodesic flow.
For this action, $m^{\BMS}$ is $A$-invariant and $m^{\BR}$ is $A$-quasi-invariant:
$(a_{-s})_* m^{\BR} =e^{(2-\delta)s} m^{\BR} $.

Set $$N:=\{ n_z=\begin{pmatrix} 1& 0\\ z&
1\end{pmatrix}: z\in \c\}\quad \text{ and }\quad  \check{N}:=\{ \check{n}_z=\begin{pmatrix} 1& z\\
0& 1\end{pmatrix}: z\in \c \}
$$ 
and for $g\in G$,  $${H}(g):=gN\quad\text{ and
}\quad \check{H}(g):=g\chN .$$

The restriction of the projection $G\to G/M$ induces a diffeomorphism
from ${H}(g)$ (resp. ${\check{H}}(g)$) to the horosphere $\H_{gM}$ (resp. $\check{\H}_{gM}$) in $\T^1(\bH^3)$ and hence
 the visual maps $u\to u^{\pm}$ induce  diffeomorphisms $P_{{H}(g)}: \partial(\bH^3)-\{g^-\} \to H(g)$ 
and $P_{{\check{H}}(g)}: \partial(\bH^3)-\{g^+\} \to \check{H}(g)$, respectively, for each $g\in G$.


 \begin{definition}\label{twom} \rm Let $y\in G$.
 \begin{enumerate}
\item  Set $$d \mu_{{H}(y)}^{\Leb}(v)=
e^{2\beta_{v^+}(o, \pi(vM))} dm_o(v^+) \text{  for $v\in H(y)$}.$$ The measure
$\mu_{{H}(y)}^{\Leb}$ is $G$-invariant: $g_*
\mu_{{H}(y)}^{\Leb}=\mu_{g({H}(y))}^{\Leb}$; in particular,
it is an $N$-invariant measure on ${H}(y)$.

 \item Set
$$d \mu_{{H}(y)}^{\PS}(v) =
e^{\delta \beta_{v^+}(o, \pi(vM))} d\nu_o(v^+) .$$   We note that $\{\mu_{{H}(y)}^{\PS}\}$ is
a $\G$-invariant family.
\end{enumerate}

\end{definition}


Fix a left $G$-invariant and right $K$-invariant metric on $G$ which
induces the hyperbolic distance $d$ on $G/K$.


\begin{notation} \begin{enumerate}
\item For $\rho>0$ and a subset $Y$ of $G$, we denote by $Y_\rho$ the
intersection of $Y$ and the $\rho$-ball centered at $e$ in $G$.
\item The $M$-injectivity radius $\rho_x$ at $x\in X$ is the supremum of $\rho$ such that
for $B_\rho:={\chN_\rho}A_\rho M N_\rho$, 
 the map $B_\rho \to x_0B_{\rho}$ given by $g\to x_0g$ is injective.\end{enumerate} \end{notation} 

\begin{definition} \label{box}\rm  A box in $X$ (around $x_0$) refers to a subset of the
form $$x_0B_\rho=x_0{\chN_\rho}A_\rho M N_\rho$$ with $0<\rho<\rho_{x_0}$ for some $x_0\in X$.
Note that $x_0B_\rho$ coincides with $x_0{\chN_\rho}A_\rho N_\rho M$.
We call this box a BMS box if
  $x_0^{\pm}\in \Lambda(\G)$, i.e., if $x_0$ belongs to the support of the BMS measure.
\end{definition}

We fix a box $x_0B_\rho$.   Set
$\tilde T_\rho:={\chN_\rho} A_\rho $ and  $T_\rho:={\chN_\rho} A_\rho  M$.
Since the measures $m^{\BMS}$ and
$m^{\BR}$ have the same transverse measures for the unstable
horospherical foliations, we have for any $\psi\in
C(x_0B_{\rho})$,
 \begin{align*} m^{\BMS}(\psi)&=\int_{y\in x_0\tilde T_\rho, m\in M }
 \int_{n\in N_\rho}\psi(yn m)d\mu_{{H}(y)}^{\PS}(yn) d\tilde\nu_{x_0\tilde T_\rho}(y)dm\\
 &=
\int_{y m\in x_0\tilde T_\rho M }
 \int_{n\in N_\rho}\psi(ym n)d\mu_{{H}(ym)}^{\PS}(ymn) d(\tilde\nu_{x_0\tilde T_\rho}\otimes m)(ym);
 \end{align*}
  \begin{align*} m^{\BR}(\psi)&=\int_{y\in x_0 \tilde T_\rho,m\in
M}\int_{N_\rho}\psi(ynm )d\mu_{{H}(y)}^{\Leb}(yn)
d\tilde\nu_{x_0 \tilde T_\rho}(y)dm\\
&=\int_{y m\in x_0\tilde T_\rho M }
 \int_{n\in N_\rho}\psi(ym n)d\mu_{{H}(ym)}^{\Leb}(ymn) d(\tilde\nu_{x_0\tilde T_\rho}\otimes m)(ym)
 \end{align*} 
that is, $d\nu_{x_0T_\rho}:=d\tilde \nu_{x_0\tilde T_\rho}\otimes dm $ denotes the transverse
measure of $m^{\BMS}$ (and hence of $m^{\BR}$) on
$x_0T_\rho$.

The following easily follows from Theorem \ref{zd}: 
\begin{cor} \label{boxzero}
 If $\G$ is geometrically finite and Zariski dense, and $E$ is a box in $X$,
then $m^{\BR}(\partial(E))=0$.
\end{cor}

\subsection{BR measure in the Iwasawa coordinates $G=KAN$}
The canonical map $\iota: {\chN} \to G/MAN =K/M$ has a diffeomorphic image  $S:=\iota({\chN})$ which is $ K/M$ minus a single point.
By abuse of notation, we use the same notation $\nu_o$ for the measure on $K$ which is the
 trivial extension of the PS measure $\nu_o$ on $ \mathbb S^2=K/M$: for $\psi\in C(K)$,
$$\int_K \psi \; d\nu_o=\int_M\int_S \psi(sm)\;d\nu_o(sX_0^-) dm$$
where $dm$ is the probability Haar measure of $M$.
The lift of the BR measure $\tilde m^{\BR}$ on $G$ can also be written as follows (cf. \cite{OS}): for $\psi\in C_c(G)$,
$$\tilde m^{\BR}(\psi)=\int_{G} \psi(k  a_sn_z) e^{-\delta s} d\nu_o(k) ds\;dz$$
where 
$ka_sn_z\in KAN$, $ds$ and $dz$ are some fixed Lebesgue measures on $\br$ and $\c$ respectively.
As usual, this means that for $\Psi(\G g)=\sum_{\gamma\in \G} \psi(\gamma g)$ with $\psi\in C_c( G)$,
$m^{\BR}(\Psi)=\tilde m^{\BR}(\psi)$.

\subsection{BR measure associated to a general unipotent subgroup}\label{gu}
A horospherical subgroup $N_0$ is a maximal unipotent subgroup  of $G$,
or equivalently, $N_0=\{g\in G: b^n g b^{-n} \to e\text{ as $n\to \infty$}\}$
for a non-trivial diagonalizable element $b\in G$. Since $A$ normalizes $N$, it follows from the Iwasawa decomposition $G=KAN$ that
 any horospherical subgroup $N_0$ is of the form $k_0^{-1}Nk_0$ for some $k_0\in K$.
 The BR measure associated to $N_0$ is defined to be
$$m_{N_0}^{\BR}(\psi):=m^{\BR}(k_0.\psi)$$
where $\psi\in C_c(X)$ and $k_0.\psi(g)=\psi(gk_0)$.
As $m^{\BR}$ is $M=N_K(U)$-invariant (here $N_K(U)$ being the normalizer of $U$ in $K$),
 this definition does not depend on the choice
of $k_0\in K$. 
If $U_0$ is a one-parameter unipotent subgroup of $G$, its centralizer $C_G(U_0)$ in $G$ is a horospherical subgroup.
The $\BR$ measure associated to $U_0$ means $m^{\BR}_{N_0}$ for $N_0=C_G(U_0)$.


\subsection{Mixing of frame flow and its consequences}
Some of important dynamical properties of flows on $X$ have been established
  only under the finiteness assumption of the BMS measure. Examples
of groups with finite BMS measure include all geometrically finite groups \cite{Sullivan1979} but not limited to those
(see \cite{Peigne2003}).
Roblin showed that if $|m^{\BMS}|<\infty$, then 
$\G$ is of divergence type.  In the following two theorems, we consider the groups $\G$ with
$|m^{\BMS}|<\infty$. We normalize $\nu_o$
so that $|m^{\BMS}|=1$.

The following two theorems were proved by \cite{Wi}, based on the the previous works of
Babillot \cite{Bab}, Roblin \cite{Roblin2003}, and Flaminio-Spatzier \cite{FS}.

\begin{thm} \cite{Wi} \label{fm} Suppose that $\G$ is Zariski dense and
$|m^{\BMS}|=1$. 
\begin{enumerate}
\item The frame flow on $X$ is mixing with
respect to $m^{\BMS}$, that is, for any $\psi_1,\psi_2\in
L^2(X, m^{\BMS})$,  as $s\to \pm \infty$,
$$\int_{X} \psi_1(xa_s)\psi_2(x)\, dm^{\BMS}(x) \to
{ m^{\BMS}(\psi_1)m^{\BMS}(\psi_2) }.$$
 \item The $\BR$ measure $m^{\BR}$ on $X$  is $N$-ergodic.
\item If $\G$ is geometrically finite, $m^{\BR}$ is the only $N$-ergodic measure on $X$ which
is not supported on a closed $N$-orbit. \end{enumerate}
\end{thm}


\begin{thm} \cite{Wi}\label{rm}
Let $\G$ be Zariski dense and $|m^{\BMS}|=1$. Then
for all $\psi_1,\psi_2\in C_c(X)$ or for $\psi_1=\chi_{E_1},\psi_2=\chi_{E_2}$  
where $E_i\subset X$ is a bounded Borel subset with
$m^{\BMS}(\partial(E_i))=0$, we have: as $s\to + \infty$,
$$\int_{X} \psi_1(xa_{-s})\psi_2(x)\, dm^{\BR}(x) \to
{ m^{\BMS}(\psi_1)m^{\BR}(\psi_2)}.$$ 
\end{thm}
We note that by the quasi-invariance of the BR measure,
$$\int_{X} \psi_1(xa_{-s})\psi_2(x)\, dm^{\BR}(x) =
e^{(2-\delta )s} \int_{X} \psi_1(x)\psi_2(xa_s)\, dm^{\BR}(x) .$$

In particular, the above theorem implies that if $\delta<2$,
$$\int_{X} \psi_1(x)\psi_2(xa_s)\, dm^{\BR}(x)\to 0\quad\text{as $s\to +\infty$}.$$

\begin{lem}\label{Zd} If $\Gamma$ is a discrete subgroup of $G$ with $\delta>1$,
then $\Gamma$ is Zariski dense in $G$.
\end{lem}
\begin{proof}
 Let $G_0$ be the identity component of the Zariski closure of $\Gamma$.
Suppose $G_0$ is a proper subgroup of $G$.
Being an algebraic subgroup of $G$, $G_0$ is contained either in a parabolic
subgroup of $G$ or in a subgroup isomorphic to $\PSL_2(\br)$. In either case,
the critical exponent of $G_0$ is at most $1$. This leads to a contradiction and hence
$G_0=G$.
\end{proof}


\section{Weak convergence of the conditional
of $\mu_{E,s}^{\BR}$}~\label{sec;cond}
In this section, we suppose that $\G$ is a Zariski dense discrete subgroup of $G$ admitting a finite
$\BMS$ measure, which we normalize so that $|m^{\BMS}|=1$. 

Fix a bounded $M$-invariant
Borel subset $E\subset X $ with $m^{\BR}(E)>0$ and $m^{\BR}(\partial (E))=0$.

For each $s>0$, define
a Borel measure $\mu_{E,s}^{\BR}$ on $X$ to be
the normalization of the push-forward $(a_{-s})*m^{\BR}|_E$: for
$\Psi\in C_c(X)$,
$$\mu_{ E,s}^{\BR}(\Psi):=\frac{1}{m^{\BR}(E)} \int_E \Psi(x a_{-s})
\; d m^{\BR}(x).$$

Equivalently, 
$$\mu_{E,s}^{\BR}(\Psi)=\frac{e^{(2-\delta)s}}{m^{\BR}(E)}  \int_{X} \Psi(x)\chi_E(x
a_{s})dm^{\BR}(x).$$
Note that $\mu_{E,s}^{\BR}$ is a probability measure supported in
the set $Ea_{-s}$.

The
following is immediate from Theorem \ref{rm}:
\begin{thm}\label{wc1} As $s\to +\infty$, $\mu_{E,s}^{\BR}$ weakly converges
to $m^{\BMS}$, that is, for any $\Psi\in C_c(X)$, 
$$\lim_{s\to +\infty} \mu_{E,s}^{\BR}(\Psi)=m^{\BMS}(\Psi).$$
\end{thm}

For simplicity, we will write for $x\in X$,
$$d\lambda_x(n)=d\mu^{\Leb}_{{H}(x)}(xn)\; \;\text{and}\;\; d\mu_x^{\PS}(n) =d\mu^{\PS}_{{H}(x)}(xn)$$
so that $\lambda_x$ and $\mu_x^{\PS}$ are
respectively the conditional measures of $m^{\BR}$ and $m^{\BMS}$ on $xN$.

Recall that $\rho_x$ denotes the injectivity radius at $x$.
\begin{definition} \label{lex} Fix $x\in X$.
For $s>0$, define a Borel measure $\lambda_{E,x,s}$ on $xN_{\rho_x}$ as follows:
for $\psi\in C_c(xN_{\rho_x})$,
$$\lambda_{E,x,s}(\psi):=\frac{
e^{(2-\delta)s}}{m^{\BR}(E)} \int_{n\in N_{\rho_x}}\psi(x n )\chi_E(x na_s)
d\lambda_{x}(n).$$
\end{definition}

Recall the notation $T_\rho={\chN_\rho}A_\rho M$ for $\rho>0$.
Let $x_0\in X$ and let $0<\rho\leq\rho_{x_0}.$ 
For any box $x_0B_{\rho}=x_0T_\rho N_{\rho}$ and
$\Psi\in C(x_0B_\rho)$, we have
\begin{align*}
\mu_{E,s}^{\BR}(\Psi)& =\frac{
e^{(2-\delta)s}}{m^{\BR}(E)} \int_{x\in X}\Psi(x)\chi_E(x a_s)dm^{\BR}(x)\\
& = 
\frac{e^{(2-\delta)s}}{m^{\BR}(E)} \int_{x \in x_0 T_\rho}\int_{n\in N_\rho}\Psi(x n
)\chi_E(x na_s) d\lambda_{x}(n)d\nu_{x_0T_\rho}(x) \\ &=
 \int_{x\in x_0 T_\rho} \lambda_{E, x,s}(\Psi|_{xN_{\rho}}) d\nu_{x_0T_\rho}(x).\end{align*}
Hence $\lambda_{E,x, s}$ is precisely the conditional measure of
$\mu_{E,s}^{\BR}$ on $xN_\rho$.

The aim of this section is to prove:

\begin{thm}\label{cw} Suppose that $x^-\in \Lambda(\G)$ and $0<\rho<\rho_x$.
 For any $\psi\in C_c(xN_{\rho})$,
$$\lambda_{E,x,s}(\psi)\longrightarrow \mu_{x}^{\PS}(\psi)\quad\text{as $s\to +\infty$}.$$
\end{thm}
The condition $x^-\in \Lambda(\G)$ is needed to approximate the measure $\lambda_{E,x,s}$ by
its thickening in the transverse direction.


 For a function $\Psi$
on $X$ and $\e>0$,
 we define functions on $X$ as follows:
 $$\Psi^+_{\e}(y):=\sup_{g\in \mathcal O_{\e}} \Psi(yg)\quad\text{and}\quad  
 \Psi^-_{\e}(y):=\inf_{g\in \mathcal O_{\e}} \Psi(yg)$$ where $\mathcal O_{\e}$ is a symmetric
$\e$-neighborhood of $e$ in $G$. We also set
$$E^+_{\e}:=E \O_{\e}\quad\text{and}\quad  E^-_{\e}=\cap_{u\in \O_\e} Eu .$$

\begin{lem}\label{r1}
Let $x\in X$ and $0<\rho<\rho_x$.  For all small $\e>0$,
there exists $\e_1>0$  such that for any non-negative 
$\Psi \in C(xT_{\e_1}N_{\rho})$ and any
$t\in T_{\e_1}$, we have
$$e^{-\e} \lambda_{x}(\Psi^-_{\e})\le
\lambda_{xt}(\Psi)\le e^{\e} \lambda_{x}(\Psi^+_{\e}).$$
\end{lem}
\begin{proof} Let $0< \e < \rho_x-\rho$.
Consider
the map $\phi_t: xN\to xtN$ given by
$\phi_t(xn)=xtn$, so that $\phi_t^*\lambda_{xt} =\lambda_x$.
Since $\phi_t$ is a translation by $n^{-1}tn$,
there exists $\e_1>0$ such that for all
$n\in N_{\rho}$ and $t\in T_{\e_1}$,  $n^{-1} t n\in \mathcal O_\e$ and
 the Radon-Nikodym derivative
satisfies $\textstyle e^{-\e}\le \frac{d{\lambda_{xt}}}{d\lambda_x}(n) \le e^\e $. 

Therefore
$$ \lambda_{xt}(\Psi) =
\int \Psi(xn (n^{-1} tn))  d\lambda_{xt}(n)
 \le e^\e \int \Psi^+_\e (xn)  d\lambda_x(n)=
e^\e \lambda_x (\Psi^+_\e) . $$
The other inequality follows similarly.  \end{proof}

\begin{lem}\label{r4}  Let $x\in X$ and $0<\rho<\rho_x$.  For any $\e>0$,
there exists $\e_1>0$
 such that for any non-negative $\Psi\in C(xT_{\e_1}N_{\rho})$,
 any $t\in T_{\e_1}$ and any $s>0$,
$$e^{-\e} \lambda_{E_{\e}^-,x,s}(\Psi^-_{2\e})\le
\lambda_{E,xt,s}(\Psi) \le e^{\e} \lambda_{E_{\e}^+,x,s}(\Psi^+_{2\e}). $$ 
\end{lem}

\begin{proof} Let $\e_1>0$ be as in Lemma \ref{r1}. We may also assume that
 $n \O_{\e_1} n^{-1}\subset
\O_{\e}$ for all $n\in N_\rho$.

For $t=\bigl( \begin{smallmatrix}
        \alpha & w\\ 0 &\alpha^{-1}
       \end{smallmatrix}\bigr)\in {\chN}AM$  and $n_z\in N$ with $\alpha+zw\ne 0$,
define $$\psi_t(z)=\tfrac{z^2w+\alpha z-\alpha^{-1} z}{\alpha+zw}\quad\text{and}\quad  
b_{t,z}=\bigl( \begin{smallmatrix}
 \alpha + zw& w\\ 0& \psi_t(z) w+\alpha^{-1}-zw
\end{smallmatrix}\bigr) .$$
Then by a direct computation, we verify that
 \be \label{com}tn_z= n_{z+\psi_t(z)} b_{t,z}.\ee

Therefore we may assume that $\e_1>0$ is small enough so that
 for all $t\in T_{\e_1}$ and $n_z\in N_\rho$,
we have $\{n_{\psi_t(z)}: n_z\in N_\rho\} \subset N_\e$, $b_{t,z}\in T_\e$, and
the absolute value of the Jacobian of the map $\psi_t|_{N_\rho}$ is at most $\e/2$.

We observe that $n_za_s= n_{z+\psi_t(z)} a_s (a_{-s} b_{t,z} a_s)$
and since the conjugation by $a_{-s}$ contracts ${\chN}A$ for $s>0$,
we have $n_za_s\in n_{z+\psi_t(z)} a_s \O_\e M$.

Since $E$ is $M$-invariant,
we deduce that $\chi_E(xtn_z a_s)\le \chi_{E^+_\e}(xn_{z+\psi_t(z)} a_s)$
for all $t\in T_{\e_1}$ and $n_z\in N_\rho$.
Together with Lemma \ref{r1}, we now obtain that for any $t\in T_{\e_1}$,
\begin{align*} \lambda_{E, xt,s}(\Psi)&= e^{(2-\delta)s} \int_{n_z\in
N_\rho}\Psi(xt n_z )\chi_E(xt n_za_s)
d\lambda_{xt}(z)\\
&\le e^{(2-\delta)s} \int_{n_z\in
N_\rho}\Psi^+_{\e}(xn_z)\chi_{E^+_{\e}}(xn_{z+\psi_t(z)} a_s)
d\lambda_{xt}(z)\\
&\le  e^\e e^{(2-\delta)s} \int_{n_z\in
N_\rho}\Psi^+_{\e}(xn_{z})  \chi_{E^+_{\e}}(xn_{z+\psi_t(z)}
a_s) d\lambda_{x}(z)
\\&\le
e^{2\e} e^{(2-\delta)s} \int_{n_z\in N_{\rho+\e}}\Psi^+_{\e}(xn_{z-\psi_t(z)})
\chi_{E^+_{\e}}(xn_{z} a_s) d\lambda_{x}(z)\\
&\le e^{2\e} e^{(2-\delta)s} \int_{n_z\in N_{\rho+3\e}}\Psi^+_{2\e}(xn_{z})
\chi_{E^+_{\e}}(xn_{z} a_s) d\lambda_{x}(z)\\
&= e^{2\e} \lambda_{E_{\e}^+,x,s}(\Psi^+_{2\e})
\end{align*}
where the last inequality follows since $N_{\rho+3\e}$ contains
$xN\cap \text{supp}(\Psi_{2\e}^+)$. The other inequality can be
proven similarly.
\end{proof}

Theorem \ref{cw} follows from:
\begin{thm}\label{f} Let $x^{-}\in \Lambda(\G)$ and $\rho<\rho_x$.
Let $\psi\in C(xN_{\rho})$ be a
non-negative function. For $\e>0$, there exists $s_0\gg 1$ such
that for any $s> s_0$,
$$e^{-4\e} \lambda_{E,x,s}(\psi)\le
\mu_{x}^{\PS}(\psi)\le e^{4\e} \lambda_{E,x,s}(\psi).$$
 Moreover, if $x^+\in \Lambda(\G)$ and $\psi$ is positive, then the 
above integrals are all non-zero.
\end{thm}
\begin{proof}
Let $\e_1$ be as in Lemma \ref{r4}.
 We note that as $x^-\in \Lambda(\G)$, $\nu(xT_{\e_1})>0$. Hence
 there exists a non-negative continuous function $\phi\in C(xT_{\e_1})$
 with $\nu(\phi)=1$.
 Define $\Psi\in C(xT_{\e_1}N_\rho)$ by
$$\Psi(xtn):=\psi(xn)\phi(xt)\quad\text{ for $xtn\in xT_{\e_1}N_{\rho}$}.$$

 Set
$\psi^+_{\e}(xn)=\sup_{u\in N_{\e}} \psi(xnu)$ and
$\psi^-_{\e}(xn)=\inf_{u\in N_{\e}} \psi(xnu)$.
Then by Lemma \ref{r4},
\begin{align*}
\mu_{E,s}^{\BR}(\Psi)&=\int_{xT_{\e_1}}\lambda_{E,xt,s}(\Psi)d\nu_{xT_\rho}(xt)\\
&\le e^{\e} \int_{xT_{\e_1}}   \lambda_{E_{\e}^+,x,s} (\psi^+_{2\e})\phi(xt) d\nu_{xT_\rho} (xt)
\\ &=e^{\e} \lambda_{E_{\e}^+,x,s}(\psi^+_{2\e}) .
\end{align*}

We can prove the other inequality similarly and hence
\be \label{br2} e^{-\e} \lambda_{E_{\e}^-,x,s}(\psi^-_{2\e})  \le
\mu_{E,s}^{\BR}(\Psi)\le e^{\e}\lambda_{E_{\e}^+,x,s}(\psi^+_{2\e}) 
.\ee

 Since the map
 $t\mapsto \mu_{xt}^{\PS}$ is continuous by \cite[Lemma 1.16]{Roblin2003},
  we have $$e^{-\e} \mu_x^{\PS}(\psi) \le m^{\BMS}(\Psi)\le e^\e \mu_x^{\PS}(\psi) $$
   by replacing $\e_1$ by a smaller one if necessary.

 Since 
$ \mu_{E,s}^{\BR}(\Psi)\to m^{\BMS}(\Psi)$ by Theorem \ref{wc1},
we deduce from \eqref{br2} that there exists $s_0>1$ such that
for all $s>s_0$, 
\be \label{inm} e^{-2\e}
\lambda_{E_{\e}^-,x,s}(\psi^-_{2\e})\le \mu_x^{\PS}(\psi)\le
e^{2\e} \lambda_{E_{\e}^+,x,s}(\psi^+_{2\e}).\ee

We claim that 
\be \label{one1}e^{-4\e}\lambda_{E,x,s}(\psi) \le
\lambda_{E_{\e}^\pm,x,s}(\psi^\pm_{2\e})\le
e^{4\e}\lambda_{E,x,s}(\psi) \ee
which will complete the proof of the theorem by \eqref{inm}.

We can deduce from \eqref{br2} that
$$e^{-\e}\mu_{E_{2\e}^-,s}^{\BR}(\Psi_{4\e}^-)\le  
\lambda_{E_{\e}^-,x,s}(\psi_{2\e}^-) \le
\lambda_{E_{\e}^+,x,s}(\psi_{2\e}^+) \le
e^{\e}\mu_{E_{2\e}^+,s}^{\BR}(\Psi_{4\e}^+).$$

 Since it follows from Theorem \ref{wc1} that
$$ e^{-\e} {\mu_{E,s}^{\BR}(\Psi)}\le
 {\mu_{E_{2\e}^{\pm},s}^{\BR}(\Psi_{4\e}^{\pm})}\le e^\e {\mu_{E,s}^{\BR}(\Psi)}
\;\; \text{for all large $s\gg 1$,}$$
 we have
$$e^{-2\e} \mu_{E,s}^{\BR}(\Psi)\le \lambda_{E_{\e}^{-} ,x,s}(\psi_{2\e}^{-})\le
 \lambda_{E_{\e}^{+} ,x,s}(\psi_{2\e}^{+})
\le e^{2\e} \mu_{E,s}^{\BR}(\Psi).$$
This implies $$\lambda_{E_{\e}^{+} ,x,s}(\psi_{2\e}^{+})\le e^{4\e}
 \lambda_{E_{\e}^{-} ,x,s}(\psi_{2\e}^{-})$$ and hence \eqref{one1} follows.
\end{proof}

\section{PS density and its non-focusing property when $\delta>1$}\label{s;PS-non-fuc}
Let $\G$ be a (non-elementary) convex cocompact subgroup of $G$. 

The assumption on $\G$ being convex cocompact is crucial for
the following theorem:

\begin{thm} \label{ub}
 For any compact subset $F_0$ 
of $X$, there exists $c_0=c_0(F_0)>1$ such that for any
$x\in F_0$ with $x^{+}\in \Lambda(\G)$ and for all $0<r\ll 1$,
$$c_0^{-1} r^{\delta}  \le \mu_{H(x)}^{\PS}(xN_r)\le c_0
r^{\delta}$$
where $xN_r=\{xn_z: |z|<r\}$. 

Similarly,  for any
$x\in F_0$ with $x^{-}\in \Lambda(\G)$ and for all $0<r\ll 1$,
we have $$ c_0^{-1} r^{\delta}  \le \mu_{\check{H}(x)}^{\PS}(x\chN_r)\le c_0
r^{\delta}$$ for $x\chN_r=\{x\chn_z:|z|<r\}$. 
\end{thm}
\begin{proof} 
 As $F_0$ is compact, up to uniform constants,
 $\mu_{{H}(x)}^{\PS}(xN_r)\asymp \nu_o(B(x^+,r))$
where $B(x^+,r)$ is the  ball around $x^+$ of radius $r$ in $\partial(\bH^3)$ in the spherical metric.
As $x^+\in \Lambda(\G)$, the above result is then due to Sullivan \cite{Sullivan1984}
who says $\nu_o(B(\xi,r))\asymp r^\delta$ uniformly for all $\xi\in \Lambda(\G)$ and for all small $r>0$
for $\G$ convex cocompact.
\end{proof}

\begin{lem}\label{c;general-pos} 
Let $\delta>1$ and $F_0\subset X$ be a compact subset.
For every $\e>0$, there exists a positive integer $d=d(\e, F_0)$ such that
for any $x\in F_0$ with $x^-\in \Lambda(\G)$ and for all small $0<r \ll
1$, we have
\[\mu_{{\check{H}}(x)}^{\PS}\{x\chn_z: |z|<r,\; |\Im(z)|\leq \tfrac{|{\Re}(z)|}d \}\le \e \cdot \mu_{{\check{H}}(x)}^{\PS}(x{\chN_r}).\]
\end{lem}

\begin{proof}
Let $r$ be small enough to satisfy Theorem \ref{ub}. For an
integer $d\ge 1$, consider $$\mathcal B_d(x, r):=\{x\chn_z: |z|<r, 
{|\Im(z)|}<r/d\}$$ which clearly contains the set in question.
Theorem~\ref{ub} implies that
\[\mu_{{\check{H}}(x)}^{\PS}(\mathcal B_d(x,r))\leq c_0 \tfrac{d\cdot r^\delta}{d^\delta}=c_0 d^{1-\delta} r^\delta,\]
where $c_0>1$ is an absolute constant independent of $d$ and $r$.

Let $d=d(\e)\gg 1$ be such that $c_0 d^{1-\delta}<c_0^{-1}\e$. Then
$\mu^{\PS}_{{\check{H}}(x)}(\mathcal B_d(x,r))\leq \e \cdot \mu^{\PS}_{{\check{H}}(x)}(x{\chN_r}),$
implying the claim.
\end{proof}


\begin{lem}\label{bzero}
 There exists $b_0> 1$ such that for all small $0<\eta\ll 1$
$$N_{b_0^{-1}\eta }A_{b_0^{-1}\eta}\chN_{b_0^{-1} \eta} M \subset \chN_\eta A_\eta N_\eta M
\subset  \chN_{b_0\eta} A_{b_0\eta} N_{b_0\eta} M .$$
\end{lem}

\begin{proof}
The claim follows since the product maps  ${\chN}\times A\times N\times M\to G$ by $(\chn,a,n,m)\mapsto {\chn}anm$
and $N\times A\times {\chN}\times M\to G$  by $(n,a,{\chn},m)\mapsto na{\chn}m$ are local diffeomorphisms at the identity.
\end{proof}

We will use the above results to prove the following proposition \ref{p;f-goodpts}. The
proof is elementary and  is based on the fact we have a good control of the conditional measures
 on contracting leaves, i.e., ${\chN}$-orbits. However, 
 the fact that this statement holds is quite essential to our approach. Indeed,
 as we explained in the introduction, one major difficulty we face is that the return times for our $U$-flow 
do not have the regularity one needs in order to get the required ergodic theorem on the nose. 
In our version of the window theorem, the set where a
window estimate holds depends on time; 
see Section~\ref{sec;window}, and in particular Theorem~\ref{wtpf} below. 
Usually in arguments with similar structure as ours, 
this fact is fatal as one has very little control on the structure 
of the ``generic'' set for the measure in question. 
In our case however the following proposition 
saves the day and provides us with a rather strong control.

In the following proposition we fix 
a BMS box $E=x_0{\chN_\rho} A_\rho N_\rho M $ with $x_0^{\pm}\in \Lambda(\G)$ and $0<\rho  <\tfrac{1}{b_0}\inf_{x\in \Omega} \rho_x $ 
where $b_0$ is as in Lemma \ref{bzero} and $\rho_x$ is the injectivity radius at $x$.
\begin{prop}\label{p;f-goodpts} Let $\delta>1$. Fix $0<r<1$.
 There exist positive numbers
$d_0=d_0(r)>1$ and $s_0\gg 1$ such that for any  Borel
subset $F\subset E$ with $m^{\BR}(F) >r\cdot  m^{\BR}(E)$ and any $s\ge s_0$, 
there exists a pair of elements $x_s, y_s\in F$ satisfying
 \begin{enumerate}
  \item  $x_s=y_s \chn_{w_s}$ for $\chn_{w_s}\in {\chN}$,
\item    $\tfrac{1}{d_0s}\leq |w_s|\leq \tfrac{d_0}{s}$ and
\item    $|\Im(w_s)|\geq \tfrac{|\Re(w_s)|}{d_0}.$
\end{enumerate}
\end{prop}

\begin{proof}
 Let $c_0> 1$ be as in Theorem \ref{ub} where $F_0$
is the $2\rho$-neighborhood of $\Omega$.
We will write $B(z, \rho)=z{\chN_\rho}$ in this proof.
For all $x\in
x_0N_\rho A_\rho M$, $x^-=x_0^-$ and hence $x^-\in \Lambda(\G)$. Hence by Theorem \ref{ub},
\be\label{ube} c_0^{-1}\eta^\delta \le \mu_{{\check{H}}(x)}^{\PS} (B(x,\eta ))\le c_0\eta^\delta\text{ for all $0<\eta<1$. }\ee

Set $d_1:=\tfrac{ \nu (x_0N_{b_0\rho}A_{b_0\rho}M)}{m^{\BR}(E)}$
where $\nu$ denotes the transverse measure of $m^{\BR}$ on $x_0N_{\rho_{x_0}}A_{\rho_{x_0}}M$.
We claim that there exists $z\in x_0N_{b_0\rho} A_{b_0\rho} M$ with $\mu_{{\check{H}}(z)}^{\PS}(B(z,b_0\rho)\cap F)>\tfrac{r}{ d_1} $.
Suppose not; then 
\begin{multline*} m^{\BR}(F)\le \int_{z\in x_0N_{b_0\rho}A_{b_0\rho}M} \mu_{{\check{H}}(z)}^{\PS}(B(z,b_0\rho)\cap F) \; d\nu(z)\\
\le  \tfrac{r}{ d_1} \nu (x_0N_{b_0\rho}A_{b_0\rho}M)   =r\cdot m^{\BR}(E)  
\end{multline*}
which contradicts the assumption on $F$.

Set $Q:=B(z, b_0 \rho) \cap F\cap \text{supp
}(\mu_{{\check{H}}(z)}^{\PS})$ and for each $s>1$,
  consider the covering $\{B(x, s^{-1}) \subset {\check{H}}(z): x\in Q \}$ of $Q$.
 By the Besicovitch covering lemma (cf. \cite{Ma}), there exists $\kappa>0$
(independent of $s$)  and a finite subset $Q_s$ such that the corresponding finite subcover $\{B(x, s^{-1}):
x\in Q_s\}$ of $Q$ is of multiplicity at most $\kappa$.

Note that for $q>1$, by \eqref{ube},
\begin{multline*} \mu_{{\check{H}}(z)}^{\PS}
(\cup_{x\in Q_s} B(x, \tfrac{1}{qs}))  \le \kappa^2 q^{-\delta}
c_0^2 \; \mu_{{\check{H}}(z)}^{\PS} (\cup_{x\in Q_s} B(x, \tfrac{1}s )  ) \le
\kappa^2 q^{-\delta} c_0^3 b_0^\delta \rho^\delta .
\end{multline*}
Hence by taking $q\ge 1$ large so that $\kappa^2 q^{-\delta} c_0^3 b_0^\delta \rho^\delta <\tfrac{r}{ 3d_1}$,
we have $$\mu_{{\check{H}}(z)}^{\PS}
(\cup_{x\in Q_s} B(x, \tfrac{1}{qs}))  <\tfrac{r}{ 3d_1} .$$

If we set $$R(s,d):=\cup_{x\in Q_s} \{w\in B(x, \tfrac{1}{s}): |\Im(w)|\le \frac{|\Re(w)|}{d}\} ,$$
it follows from Lemma \ref{c;general-pos} that there exist $d_2>1$ and $s_0>1$  such
that for any $s>s_0$,
$$ \mu_{{\check{H}}(z)}^{\PS} (R(s,d_2)) < \tfrac{r}{ 3d_1} .$$ 

 Hence for any $s>s_0$, the set
$$Q -(\cup_{x\in Q_s} B(x, \tfrac{1}{qs} ) \cup R(s,d_2))$$
has a positive $\mu_{{\check{H}}(z)}^{\PS}$ measure (at least
$\frac{r}{ 3d_1}$). In particular, there exists $x_s\in Q$
such that $(Q \cap B(x_s, \tfrac{1}{s})) -(B(x_s, \tfrac{1}{qs}) \cup
R(s,d_2))$ has a positive $\mu_{{\check{H}}(z)}^{\PS}$ measure. Picking $y_s$
from this set, we have found a desired pair $x_s, y_s$ from $F$
with $d_0=\max (q, d_1)$.
\end{proof}

\section{Energy estimate and $L^2$-convergence for the projections}\label{sec;energy}
Let $\G$ be a convex cocompact subgroup of $G$ with $\delta >1$ and fix a BMS box $E\subset X$
(see \ref{box} for its definition). 
We have $m^{\BR}(E)>0$ and by Lemma \ref{Zd} and Corollary \ref{boxzero},
$m^{\BR}(\partial(E))=0$.
$$\text{In the entire section, we fix $x\in X$ with $x^{\pm}\in \Lambda(\G)$ and $0<\rho< \tfrac{1}{\sqrt 2}\rho_x$.}$$ 
Recall the definition of the measure $\lambda_{E,x,s}$ on $xN_{\rho_x}$ from \eqref{lex}:
  for $\psi\in C(xN_{\rho_x})$,
$$\lambda_{E, x,s}(\psi)=\frac{
e^{(2-\delta)s}}{m^{\BR}(E)} \int_{n\in  N_{\rho_x}}\psi(x n )\chi_E(x na_s)
d\lambda_{x}(n).$$

\subsection{Projections of $\mu_{{H}(x)}^{\PS}$ and $\lambda_{E,x,s}$}\label{sec;projection}
 The $N$-orbit
of $x$ can be identified with $\br^2$ via the visual map
$xn\mapsto (xn)^+\in \partial(\bH^3)-\{x^-\}$ and the
identification of $\partial(\bH^3)-\{x^-\}$ with
$\br^2$ by mapping $x^-$ to the point at infinity. Therefore we
may consider $\lambda_{E, x, s}$ and $\mu_{{H}(x)}^{\PS}$ as measures on $\br^2$.

Let $$U=\{\begin{pmatrix} 1& 0\\ t
&1\end{pmatrix}: t\in \br\};\quad V=\{\begin{pmatrix} 1& 0\\
i t &1\end{pmatrix}: t\in \br\}.$$


 In the sequel 
by a measure on $[0,2\pi]$ we mean the normalized Lebesgue measure.
For each $\theta\in [0,2\pi)$, we set
$U_\theta=m_\theta U m_\theta^{-1}$ and $V_\theta=m_\theta V
m_\theta^{-1}$.
We may identify $U_\theta$ as the line in $\br^2$ in the $\theta$-direction and $V_\theta$ as the line in the $\theta+\pi/2$ direction.

We denote by $p_\theta: U_\theta V_\theta \to V_\theta$ the projection
parallel to the line $U_\theta$. For $\tau>0$, set
 $$U_\theta^\tau:=\{t\exp(i\theta):t\in
[-\tau, \tau] \}\text{ and }\;
 V_\theta^\tau:=\{it\exp(i\theta):
t\in [-\tau, \tau] \} .$$

\begin{definition}\label{ss}\label{ss2} \rm Fix $0<\theta<\pi$,
$0<\tau\le \rho$ and $s>1$.  We define the measures 
on $xV_\theta^\tau$ as follows: for $\psi\in C_c(xV_\theta^\tau)$,
$$ \sigma_{x,\theta}^\tau(\psi)
:=\int_{xV_\theta^\rho U_\theta^\tau} \psi (p_\theta (y)) \; d\mu_{{H}(x)}^{\PS} (y),$$ and 
$$ \sigma_{x,\theta,s}^\tau(\psi)
:=\int_{xV_\theta^\rho U_\theta^\tau} \psi (p_\theta (y)) \; d\lambda_{E, x,s}(y).$$ That is,
$ \sigma_{x,\theta}^\tau$ and $\sigma_{x,\theta, s}^\tau$ are respectively
 the push-forwards of $\mu_{{H}(x)}^{\PS}|_{xV_\theta^\rho U_\theta^\tau}$ and
$\lambda_{E, x,s}|_{xV_\theta^\rho  U_\theta^\tau}$ via the map
$p_\theta$.
\end{definition}

\subsection{Energy  and Sobolev norms of
the projections}

Consider the Schwartz space
$\mathcal S:=\{f\in L^2(xV_\theta): t^\alpha  f^{(\beta)}\in L^2(xV_\theta)\},$ where $\alpha,\beta\in\mathbb{N}\cup\{0\}$
and $f^{(\beta)}$ is the $\beta$-th derivative of $f$. Denote by
$\mathcal S'$ the dual space of $\mathcal S$ with the strong dual
topology, which is the space of tempered distributions.
For $r>0$, we consider the following Sobolev space
$$H^r(xV_\theta):=\{f\in \mathcal S': (1+|t|)^{r}\hat f\in
L^2(xV_\theta)\}$$ with the norm
$$\|f\|_{2,r}:= \|(1+|t|)^r \hat f\|_{L^2(xV_\theta)}$$
 where $\hat f$ denotes the Fourier transform of $f$.


We recall the notion of $\alpha$-energy:
\begin{definition} [$\alpha$-energy] For $\alpha>0$ and a Radon measure $\mu$ on
$\br^2$, the $\alpha$-energy of $\mu$ is given by
$$I_\alpha(\mu):=\int_{\br^2}\int_{\br^2}\frac{1}{{|x-y|}^\alpha}d\mu(x)d\mu(y) .$$
\end{definition}

It is a standard fact that $I_\alpha(\mu)$ can be written as
\be \label{edic}  I_\alpha(\mu)= \alpha\int_{\br^2} \int_0^\infty
\frac{\mu(B(x,\ell ))}{\ell^{1+\alpha}} d\ell d\mu(x) \ee where $B(x,\ell)$ is the
Euclidean disc around $x$ of radius $\ell$.


The $\alpha$-energy of a measure $\mu$ is a useful tool in studying the projections of $\mu$ in various
directions. See ~\cite[Proposition 2.2]{PS} or \cite[Theorem 4.5]{Ma1} for the following theorem:
\begin{thm}   \label{t;r-n-proj}
Let $\nu$ be a Borel probability measure on $\bbr^2$ with compact
support. If the $1$-energy of $\nu$ is finite, i.e., $I_1(\nu)<\infty$, then the following hold:
\begin{enumerate}
\item $p_{\theta*}\nu$ is
absolutely continuous with respect to the Lebesgue measure for
almost all $\theta$;
\item there exists $c>1$ (independent of $\nu$) such that for any $0<r<\tfrac 12$,
\[ c^{-1}I_{1+2 r}(\nu)\leq  \int\|D(p_{\theta*}\nu )\|_{2,r}^2d\theta\leq c\; I_{1+2 r}(\nu),\]
where $D(p_{\theta*}\nu)$ is the Radon-Nikodym  derivative of $p_{\theta*}\nu$ with
respect to the Lebesgue measure. \end{enumerate}
\end{thm}

\begin{lem}\label{ec} Let $Q\subset \br^2$ be a compact subset, $c>0$ and $\beta>0$ be fixed.
 Let $\mathcal M$ be a collection of Borel measures on $Q$
such that 
\be\label{mbig} \mu(B(x, \ell)) < c\cdot \ell^\beta  \;\;\text{ for all $\mu\in \mathcal M$,
 $x\in \supp (\mu )$ and $\ell>0$}.\ee
Then for any $0<\alpha < \beta$,
$$\sup_{\mu\in \mathcal M} I_\alpha(\mu) <\infty.$$
\end{lem}

\begin{proof}
 Fix $ 0<\alpha <\beta$. We use \eqref{edic}. Note that since $\mu(B(x,\ell))\le \mu(Q)$, \eqref{mbig} has meaning only when
 $\ell$ is not too big. We use \eqref{mbig} only for $0<\ell <1$ and use the upper
 bound of $\mu(Q)$ for $\ell \ge 1$.
We have
\begin{align*}\tfrac{1}\alpha I_\alpha(\mu)&=
\int_{Q}\int_0^{1} \frac{\mu(B(x,\ell))}{\ell^{1+\alpha}}d\ell d\mu(x) + 
\int_{Q}\int_{1}^\infty \frac{\mu(B(x,\ell))}{\ell^{1+\alpha}}d\ell d\mu(x)
\\ &\leq  c\cdot  \ell^{\beta-\alpha}|_{\ell=0}^{\ell=1}\cdot  \mu(Q)+ \ell^{-\alpha}|_{\ell=1}^{\ell=\infty} \cdot \mu(Q)^2 
\\ & =\mu(Q) (c+\mu(Q)) .
\end{align*}
Now, since $Q$ is compact, the assumption implies
that $\sup_{\mu\in \mathcal M}\mu (Q)<\infty$.
Hence $I_\alpha(\mu)$ is uniformly bounded for all $\mu\in \mathcal M$.
\end{proof}

\begin{cor}  \label{c;ps-leb} 
Fix $0<\tau\le \rho$. The following holds for almost all
$\theta$:
\begin{enumerate} 
\item $\sigma^\tau_{x,\theta}$ is absolutely continuous with
respect to the Lebesgue measure on $xV_{\theta}^\tau$;
\item its support has a positive   Lebesgue measure;
\item its Radon-Nikodym derivative
satisfies $D(\sigma^\tau_{x,\theta})\in H^{r}(xV_\theta^\tau)$ for any $0<r<\tfrac{\delta-1}2$.
\end{enumerate}
\end{cor}
\begin{proof} 
It follows from Theorem \ref{ub} and Lemma \ref{ec} that
for any $0<\alpha<\delta,$ $I_\alpha(\mu_{{H}(x_0)}^{\PS}|_{x_0N_\rho})<\infty$.

Now, the fact that the support of the projection has positive measure follows from
\cite[Theorem I]{Mars}.
The other two claims follow from Theorem \ref{t;r-n-proj}.
\end{proof}

$$\text{ We fix  $0<r<\tfrac{\delta-1}2 $ for the rest of this section.}$$

\begin{terminology}[$\op{PL}$-direction] If $\theta$ satisfies Corollary \ref{c;ps-leb} with respect to $r$, we will call $\theta$
 as a ``$\op{PL}$" direction for
$(x,\tau)$, or simply for $\tau$ when $x$ is fixed.\end{terminology}

\subsection{Uniform bound for the energy of $\lambda_{E,x,s}$, $s\ge 1$}
In this subsection, we
set 
$$\lambda_{E,x,s}^\dag:=\lambda_{E,x,s}|_{xN_\rho}.$$ 
 We will show that the collection $\mathcal M=\{\lambda_{E,x, s}^\dag: s\ge 1\}$ 
of measures on $xN_{\rho}$
satisfies the hypothesis of Lemma \ref{ec} with $\beta=\delta$.
We may consider $\lambda_{E,x,s}^\dag$ as a measure on $\br^2$ supported on the $\rho$-ball around the origin.

Since $E$ is a BMS box, $E$ is of the form $x_0N_{r_0}^-A_{r_0} N_{r_0} M$ for some $0<r_0<\rho_{x_0}$
where $x_0^{\pm}\in \Lambda(\G)$.

\begin{lem}\label{lb} For all $s\ge 1$, we have
$$xN_\rho\cap Ea_{-s}\subset \{xn\in xN_\rho: d(xn,
P_{{H}(x)}^{-1}(\Lambda(\G)-\{x^-\}))\le e^{-s} r_0\}$$
where $P_{{H}(x)}: \partial_\infty(\bH^3)-\{x^-\} \to
{H}(x)$ is defined in the subsection \ref{ph} 
and $d$ denotes the
Euclidean distance: $d(xn_z, xn_{z'})=|z-z'|$.
\end{lem}

\begin{proof} 
Suppose $xn\in E$, so that
$xn=x_0{\chn}_wa_tn_z m_\theta a_{-s}$ with $|z|<r_0$. We may write it as
$$xn=x_0{\chn}_wa_{t-s} m_\theta n_{e^{-s}e^{2\pi i \theta} z}.$$
If we set $y:=x_0{\chn}_wa_{t-s}m_\theta$, then $y^+=x_0^+$. Hence
$y^+\in \Lambda(\G)$. Since $xn=yn_{e^{-s} e^{2\pi i \theta } z}$ and
$|e^{-s} e^{2\pi i \theta } z|<e^{-s} r_0$, the claim follows.
\end{proof}

\begin{thm}\label{ue}
There exists $c>0$ such that for all $s\gg 1$,
 $y\in \supp (\lambda_{E,x,s}^\dag)$ and any $ \ell>0 $,
$$\lambda_{E,x,s}^\dag (B(y, \ell)) < c\cdot \ell^\delta $$
where $B(y,\ell)=\{yn_z:|z|<\ell\}$.
\end{thm}

\begin{proof} Since $B(y, 2\rho) $ contains $xN_\rho$, it suffices to show the above
for $0<\ell<2\rho$. Since 
 $\mbox{supp}(\lambda_{E,x,s})\subset
E a_{-s}\cap xN_\rho,$ it follows from Lemma \ref{lb} that for each $z \in \mbox{supp}(\lambda_{E,x,s}^\dag)$, $B(z, 3 \rho
e^{-s})$ contains $B(w, \rho e^{-s})$ for some $w\in
{H}(z)$ with $w^+\in \Lambda(\G)$.

Hence by Theorem
\ref{ub} we have
$$ \mu_{{H}(z)}^{\PS}(B(z,3 \rho e^{-s}))\ge
 \mu_{{H}(z)}^{\PS}(B(w,\rho e^{-s}))\ge
  c_0^{-1} (3\rho)^\delta e^{-\delta s}  $$
where $c_0$ is as in Theorem \ref{ub} with $F_0$ being the $\rho$-neighborhood of $\Omega$.
Consider the covering of $\mbox{supp}(\lambda_{E,x,s}^\dag)$ given by the balls
 $B(z,3\rho e^{-s})$, $z\in \mbox{supp}(\lambda_{E,x,s}^\dag)$.
 By the Besicovitch covering lemma we can choose a finite set $J_s\subset \mbox{supp}(\lambda_{E,x,s}^\dag)$ such that
the corresponding finite collection $\{B(z,3\rho e^{-s}): z\in J_s\}$ has  multiplicity at most $\kappa$
(independent of $s$) and covers $\mbox{supp}(\lambda_{E,x,s}^\dag).$

Now we consider two cases for $\ell.$ 

{\bf Case 1.} $0<\ell\le e^{-s}.$ 

In this case, for any $y\in \mbox{supp}(\lambda_{E,x,s}^\dag)$, we
have
\[\lambda_{E,x,s}^\dag(B(y,\ell))\leq \pi e^{(2-\delta)s}\ell^2 \leq \pi \ell^{\delta}.\]

{\bf Case 2.} $e^{-s} <\ell<2\rho .$ Let $J_{y,s}=\{z\in J_s: B(z,3\rho e^{-s}) \subset
B(y,3\ell)\}$. We have
\begin{align*}\lambda_{E,x,s}^\dag(B(y,\ell)) &\leq \sum_{z\in J_s} \{\lambda_{E,x,s}(B(z,3\rho e^{-s})): B(z,3\rho e^{-s})\cap B(y,\ell)\neq\emptyset\}\\
& \leq\sum_{z\in J_{y,s}}\lambda_{E,x,s}(B(z,3\rho e^{-s})) \\ & \leq 
\sum_{z\in J_{y,s}}
e^{(2-\delta)s} (3\rho)^2 e^{-2s}\\ & \leq
 c_0 (3\rho)^{2-\delta}
 \sum_{z\in J_{y,s}} \mu_{{H}(y)}^{\PS}(B(z,3\rho e^{-s})) 
 \\ &
 \leq \kappa c_0 (3\rho)^{2-\delta}
    \mu_{{H}(y)}^{\PS}(B(y,3 \ell))
    \\ & \leq 3^\delta\kappa
 c_0^2 (3\rho)^{2-\delta} \ell^\delta .\end{align*}

Hence for all  $0<\ell <2\rho $ and $y\in \text{supp}(\lambda_{E,x,s})$,
$$\lambda_{E,x,s}^\dag(B(y,\ell))\le c_1 \ell^\delta$$
for some constant $c_1>0$ independent of $s\gg 1$. 
\end{proof}

Therefore by Lemma \ref{ec}, we deduce:
\begin{cor}\label{cor;unif-energy}
 For any $0<\alpha<\delta$,
$$\sup_{s\gg 1} I_\alpha(\lambda_{E,x,s}^\dag)<\infty .$$
\end{cor}

\subsection{$L^2$-convergence of projected measures}
$$\text{
Recall the notation $\sigma^{\tau}_{x,\theta,s_i}$ and $  \sigma^{\tau}_{x,\theta}$ from Definition~\ref{ss}.}$$
 Theorem \ref{cw} is used crucially
in the following proposition: 

\begin{prop}\label{conver}
 Fix $0< \tau \le \rho$,  a $\op{PL}$-direction
$\theta\in M$ for $(x,\tau)$ and a sequence $s_i\to
+\infty$.
If $\sup_i \|D(\sigma^{\tau}_{x,\theta,s_i} )\|_{2,r} <\infty $,
then \[D(\sigma^{\tau}_{x,\theta,s_i}) \xrightarrow{L^2(xV_\theta)}
D(\sigma^{\tau}_{x,\theta}) \quad\text{as $i\to \infty$}.\]
\end{prop}
\begin{proof}
By Theorem \ref{cw} and the assumption of $x^{\pm}\in \Lambda(\G)$,
 $\lambda_{x,s_i}|_{xU_\theta^\rho V_\theta^\tau}$
weakly converges to $\te \mu_{{H}(x)}^{\PS}|_{xU_\theta^\rho
V_\theta^\tau}$ as $s_i\to \infty$. Therefore $\sigma^{\tau}_{x,\theta,s_i}$
 weakly converges to $\sigma^{\tau}_{x,\theta}$ as $i\to \infty$.
Hence it suffices to show that the collection
$$\{ D(\sigma^{\tau}_{x,\theta,s_i})\in {L^2(xV_\theta^\tau)}\}$$ is relatively compact in
$L^2(xV_\theta^\tau)$.
Since  this collection is
uniformly bounded in the Sobolev space $H^{r}(xV_\theta^\tau)$ by the assumption, the claim follows from the fact that
we have $H^r(xV_\theta^\tau)$ embeds
compactly in $L^2(xV_\theta^\tau)$  for any $r>0$  (see~\cite[Theorem 16.1]{LM}).
\end{proof}

Recall that by a measure on $[0, 2\pi)$, we mean the Lebesgue measure normalized to
be the probability measure.
\begin{thm}\label{c;unif-l2} Let $s_i\to
+\infty$ be a fixed sequence. For any $\e>0$ and any
finite subset $\{\tau_1,\ldots,\tau_n\}$ of $(0,\rho]$,
 there exists a Borel subset
$\Theta_\e(x)\subset [0,2\pi),$ of measure at least $1-\e$, such that 
\begin{enumerate}
\item every 
$\theta\in\Theta_\e( x)$ is a $\op{PL}$ direction for $(x,\tau_\ell)$ for each $1\le \ell\le n$; 
 \item for each $\theta\in \Theta_\e(x)$, there exists an infinite subsequence  $\{s_{j_i}\}$(depending on $(x,\theta)$) such that
for each $1\le \ell\le n$, \[D(\sigma^{\tau_\ell}_{x,\theta,s_{j_i}}) \xrightarrow{L^2(xV_\theta)}
D(\sigma^{\tau_\ell}_{x,\theta}) .\] \end{enumerate}
\end{thm}

\begin{proof}
 Recall that we fixed some
$0<r<(\delta-1)/2$. By Corollary \ref{cor;unif-energy} and Theorem
\ref{t;r-n-proj}, there is a constant $L>1$ such that
\[\sup_i\int\|D(\sigma^{\tau_\ell}_{x,\theta,s_i})\|_{2,r}^2d
\theta\leq L\hspace{7mm}\mbox{for}\h\h1\leq\ell\leq n.\]
Hence using Corollary~\ref{c;ps-leb} and Chebyshev's inequality, we
deduce that for any $\e>0$, there exists some $L_0>0$ such that if we let
\[\Theta^{\tau_\ell}_{s_i}=\{\theta: \theta\h\mbox{is a $\op{PL}$ direction for $(x,\tau_\ell)$ and}
 \h\h\|D(\sigma^{\tau_\ell}_{x,\theta, s_i} )\|_{2,r}^2< L_0\},\]
then for all $i>0$, we have
$m(\Theta^{\tau_\ell}_{s_i})>1-\frac{\e}{2n}.$ Let
$\Theta_i=\cap_\ell\Theta_{s_i}^{\tau_\ell}$ and let
$\Theta=\limsup_i \Theta_i$.  Then $m(\Theta)>1-\e$.
For $\theta\in\Theta$, $\theta$ lies in infinitely many of $\Theta_i$'s, i.e.,
$\theta\in \Theta_{j_i}$ for some infinite subsequence $\{j_i\}$. 
Hence the claim follows from Proposition \ref{conver} applied to $\{s_{j_i}\}$.
\end{proof}

\subsection{Key lemma on the projections of $\lambda_{E,x,s}^\dag$}

The following is the key technical lemma in the proof of the window theorem.

\begin{lem}[Key Lemma] \label{zerotwo}  Fix
$0<\tau<\rho$ and a sequence $s_i \to +\infty $.
 For any $\e>0$, there exists a Borel subset $\Theta_\e(x)\subset [0,2\pi)$ of measure at least $1-\e$ such that
 if $\theta\in \Theta_\e(x)$ and $E_i m_\theta^{-1} \subset X$ is a sequence of Borel subsets
satisfying $$\lambda_{ E, x,s_i}(xN_\tau -E_i m_\theta^{-1})
\to 0,$$
then there is an infinite subsequence $\{s_{j_i}\}$ such that
 for any Borel subset $O_\theta(x)\subset xV_\theta^\rho$,
$$\limsup_{i} \sigma_{x,\theta,s_{j_i}}^{\rho}(O_\theta(x) \setminus p_\theta( E_i m_\theta^{-1} \cap xN_\tau))\le
\sigma_{x,\theta}^{\rho} \{ t\in O_\theta(x): D(\sigma^\tau_{x,\theta})( t)=0\}.$$
\end{lem}

By Theorem \ref{c;unif-l2}, the Key Lemma follows from the following lemma.
Observe that this is a rather strong control on the conditional measures, as
one can easily construct counter-examples in a general setting.
Here our $L^2$- convergence result of the projection measures  to a ``rich" measure
is crucially used. 
\begin{lem} Fix
$0<\tau<\rho$ and
 a $\op{PL}$ direction  $\theta\in [0,2\pi),$ simultaneously for $(x,\tau)$ and $(x,\rho)$.
 Let $W_i m_\theta^{-1} \subset x N_\tau$ be a sequence of Borel subsets
and $\{s_i\}$ be a sequence tending to infinity.
Assume the following holds as $i\to \infty$:
\begin{enumerate}
 \item  $D(\sigma^{\tau}_{ x,\theta,s_i})\xrightarrow{L^2(xV_\theta)}
D( \sigma_{x,\theta}^{\tau}) $;
\item $D(\sigma^{\rho}_{x,\theta,s_i}) \xrightarrow{L^2(xV_\theta)}
D( \sigma_{x,\theta}^{\rho} ) $; 
\item $\lambda_{ E, x,s_i}(xN_\tau -W_i m_\theta^{-1})
\to 0$.
\end{enumerate}
Then for any Borel subset $O_\theta(x)\subset xV_\theta^\rho$,
$$\limsup_i \sigma_{x,\theta,s_i}^{\rho}(O_\theta(x) \setminus p_\theta( W_i m_\theta^{-1}))\le
\sigma_{x,\theta}^{\rho} \{ t\in O_\theta(x): D(\sigma^\tau_{x,\theta})( t)=0\}.$$
\end{lem}

\begin{proof}
 Set $\mathcal P^\tau:=   \{ t\in O_\theta(x): D(\sigma^\tau_{x,\theta})( t)> 0\}$ and
 \[
 \mathcal L_{i}^\tau:=p_\theta(O_\theta(x) U_\theta^\tau\cap W_i m_\theta^{-1}) =O_\theta(x)
 \cap p_\theta( W_i m_\theta^{-1}).
 \] 
 
 For $n_0>1$, define
 $$\Sigma_{n_0}=\{t\in \mathcal P^\tau:
D(\sigma^{\tau}_{x,\theta})( t) \geq \tfrac{1}{n_0}, \;\; D(\sigma^{\rho}_{x,\theta})( t) <n_0\}.$$

Let $\e>0$ be arbitrary. There exists $n_0=n_0(\e)>1$ such that
\[\sigma_{x,\theta}^{\rho} (\Sigma_{n_0})>(1-\e)\sigma_{x,\theta}^{\rho}(\mathcal P^\tau) .\]
 Since $D(\sigma^{\tau}_{x,\theta,s_i} )\rightarrow D(\sigma^{\tau}_{x,\theta})$ in $L^2(xV_\theta)$, denoting
 by $\lambda$ the Lebesgue measure on $xV_\theta$, we have
\begin{align*} &|\sigma^\tau_{x,\theta}(O_\theta(x) \setminus
\mathcal{L}^\tau_i)-\sigma^\tau_{x, \theta,s_i}(O_\theta(x) \setminus
\mathcal{L}^\tau_i)| \\ & \le
\int_{xV_\theta}| D(\sigma^{\tau}_{x,\theta})-D(\sigma^{\tau}_{x,\theta,s_i})|d\lambda\\&
\le \| D(\sigma^{\tau}_{x,\theta})-D(\sigma^{\tau}_{x,\theta,s_i}) \|_2 \cdot 
\lambda(xV_\theta)^{1/2} \rightarrow 0 .\end{align*}

Since  
 $\sigma^\tau_{ x,\theta,s_i} (O_\theta(x) \setminus \mathcal L_i^\tau) \le
\lambda_{E, x,s_i}(xN_\rho -W_i m_\theta^{-1})
\rightarrow 0 $
by the assumption on $W_im_\theta$, 
it follows now that there is some $i_0=i_0(n_0)$ such that for all $i\geq i_0$,
$$\sigma^\tau_{x,\theta} (O_\theta(x) \setminus \mathcal L_i^\tau)<\tfrac{\e}{n_0^2}.$$ 
Note that for any set $\Upsilon\subset\Sigma_{n_0}$ with
$\sigma^\tau_{x,\theta}(\Upsilon)<\tfrac{\e}{n_0^2}$, we have $\sigma^{\rho}_{x,\theta}(\Upsilon)<\e$. To see this,
note that if 
\[
\tfrac{1}{n_0} \lambda(\Upsilon) \le \int_{\Upsilon} D(\sigma^{\tau}_{x,\theta}, t) d\lambda(t) \le \tfrac{\e}{n_0^2},
\] 
then
$\lambda(\Upsilon)\le \tfrac{\e}{n_0}$ and hence
$$\sigma^{\rho}_{x,\theta} (\Upsilon) =\int_{\Upsilon} D(\sigma^{\rho}_{x,\theta})( t) d\lambda(t) \le n_0 
\lambda(\Upsilon)\le \tfrac{\e}{n_0}\le \e.$$

 Therefore we have \begin{align*}&
 \sigma^{\rho}_{x,\theta}(O_\theta(x) \setminus \mathcal L_i^\tau)\\& \le
\sigma^{\rho}_{x,\theta}((O_\theta(x)\setminus \mathcal L_i^\tau)\cap \Sigma_{n_0} )
+\sigma^{\rho}_{x,\theta}((O_\theta(x) \setminus \mathcal L_i^\tau )\cap (O_\theta(x)\setminus \Sigma_{n_0}))\\ \notag &
\le \e +\e \cdot \sigma^{\rho}_{x,\theta}(\mathcal P_\theta^\tau) +
 \sigma^{\rho}_{x,\theta}(O_\theta(x) -\mathcal P_\theta^\tau) . \end{align*}

Since $\e>0$ is arbitrary,
\begin{equation}\label{sigmas} \limsup_i \sigma^{\rho}_{x,\theta}(O_\theta(x)\setminus \mathcal L_i^\tau) 
 \le \sigma^{\rho}_{x,\theta}(O_\theta(x) -\mathcal P_\theta^\tau) .\end{equation}

 Now
since $ D(\sigma^{\rho}_{x,\theta,s_i})\rightarrow  D(\sigma^{\rho}_{x,\theta})$ in $L^2(xV_\theta^\rho)$, we
have
\[  |\sigma^{\rho}_{x,\theta} ( O_\theta(x) \setminus \mathcal L_i^\tau )-
\sigma^{ \rho}_{x,\theta,s_i} ( O_\theta(x)\setminus \mathcal L_i^\tau)|
\leq\int_{xV_\theta}|D(\sigma^{\rho}_{x,\theta,s_i}) -D(\sigma^{\rho}_{x,\theta})| d\lambda \rightarrow0.\]
Combined with \eqref{sigmas}, this implies that 
 $$ \limsup_i \sigma_{ x, \theta,s_i}^{\rho}(O_\theta(x)\setminus \mathcal L_i^\tau )
 \le \sigma^{\rho}_{x,\theta}(O_\theta (x)-\mathcal P_\theta^\tau) .$$
\end{proof}

\section{Recurrence properties of BMS and BR measures} 
\subsection{Theorems of Marstrand on Hausdorff measures}\label{sec:Mar}
Let $\Lambda\subset \br^2$. The $s$-dimensional Hausdorff measure of $\Lambda$ is defined to be
$$\mathcal H^s(\Lambda)=\inf_{\eta\downarrow 0}\mathcal H_\eta^s(\Lambda),$$
where $\te \mathcal H^s_\eta (\Lambda):=\{\sum_i d(W_i)^s: \Lambda\subset \cup_{i=1}^\infty W_i, \;\; d(W_i)\le \eta\}$
and $d(W_i)$ denotes the diameter of $W_i$.

The Hausdorff dimension of $\Lambda$ is 
$$\text{dim}(\Lambda)=\sup\{s:\mathcal H^s(\Lambda)>0\}=\inf\{s:\mathcal H^s(\Lambda)=\infty\} .$$

A set $\Lambda$ is called an $s$-set if $0<\mathcal H^s(\Lambda)<\infty$.
Following Marstrand \cite{Mars}, a point $\xi\in \Lambda$ is called
 a condensation point for $\Lambda$ if $\xi$ is a limit point from $(\xi, \theta)\cap
\Lambda$ for almost all $\theta$ where $(\xi, \theta)$ denotes
the ray through $\xi$ lying in the direction $\theta$.

Let $\Lambda$ be an $s$-set in the following three theorems:
\begin{thm} \cite[Theorem 7.2]{Mars}\label{mar1}
If $s>1$, $\mathcal H^s$-almost all
points in $\Lambda$ are
  condensation points for $\Lambda$.
\end{thm}

\begin{thm} \cite[Lemma 19]{Mars}\label{Mars2} 
If $s>1$, almost
every lines $L$ through $\mathcal H^s$-almost all points in $\Lambda$ intersect $\Lambda$ in a set of
dimension $s -1$.
\end{thm}

\begin{thm}  \cite[Theorem II]{Mars} \label{Mars3} If $s\le 1$, then
the projections of $\Lambda$ have Hausdorff dimension $s$ for almost all
directions.
\end{thm}

\subsection{$U$-Conservativity of $m^{\BR}$} 
In the rest of this section, we assume that $\G$ is convex cocompact.

\begin{thm} \label{psh}\cite{Sullivan1984} 
For $x\in G$, the measure $\mu_{H^{+}(x)}^{\PS}$ on $xN$
is  a $\delta$-dimensional Hausdorff measure supported on the set 
$\{xn\in H^{+}(x): (xn)^{+}\in \Lambda(\G) \}.$ Furthermore, 
this is a positive and locally finite measure on $xN$.
\end{thm}

For $U=\{u_t=\bigl(\begin{smallmatrix} 1& 0\\ t
&1\end{smallmatrix}\bigr): t\in \br\}$,
we recall the definition of a conservative action:
\begin{definition}[Conservative action] Let $\mu$ be a locally finite $U$-invariant measure
on $X$. The $U$-action on $X$ is conservative for
$\mu$ if  one of the following equivalent conditions holds:
\begin{enumerate}
\item for every positive Borel function $\psi$ of $X$, $$\textstyle \int_{t\in \br} \psi(xu_t) dt
=\infty
\text{ for a.e. $x\in X$};$$
\item for any Borel subset
$B$ of $X$ with $\mu(B)>0$,
$$\textstyle \int_{t\in \br} \chi_B(xu_t) dt
=\infty\quad\text{ for a.e. $x\in B$; }$$ 

\end{enumerate}
\end{definition}

 The following 
is Maharam's recurrence theorem (cf. \cite[1.1.7]{Aa}).
\begin{lem}\label{con}
If there is a measurable subset $B\subset X$ with
$0< m^{\BR}(B)<\infty$ such that for almost all $x\in X$,
$\int_{0}^\infty \chi_B(xu_t) dt=\infty $, then $U$ is conservative for $m^{\BR}$.
\end{lem}

\begin{thm}\label{conser}
If $\delta>1$, then $U$ is conservative for $m^{\BR}$.
\end{thm}
\begin{proof} 
Recall
the notation $\Omega=\text{supp}(m^{\BMS})$ and $\Omega_{\BR}=\text{supp}(m^{\BR})$. Set
$$\mathcal F:=\{x\in X:  x^-\in \Lambda(\G), xu_{t} \notin \Omega \text{ for all large $t\gg_x 1$
}\}. $$ Hence $x\in \mathcal F$ means  $(xu_t)^+\notin \Lambda(\G)$  for all large $t\gg_x 1$.
 We claim that \begin{equation}\label{cff} m^{\BR}( \mathcal
F) =0.\end{equation} Suppose not. Then by the Fubini theorem,
there is a set $\mathcal O\subset \Omega_{\BR}$ with
$m^{\BR}(\mathcal O)>0$ such that for all $x\in \mathcal O$,
$xm_\theta\in \mathcal F$ for a positive measurable subset of
$\theta$'s. Note that $\nu_o(\{x^-\in \Lambda(\G): x\in \mathcal O\})>0$ 
 where $\nu_o$ is the
PS measure on $\Lambda(\G)$. Fix $\xi_0\notin
\Lambda(\G)$ and identify $\partial(\bH^3)-\{\xi_0\}$ with
$\br^2$. Since $\nu_o|_{\partial(\bH^3)-\{\xi_0\}}$ is equivalent
to the $\delta$-dimensional Hausdorff measure $\mathcal H^\delta$
on $\Lambda(\G)\subset \br^2$ by Theorem \ref{psh}, we have $\mathcal H^\delta \{x^-\in
\Lambda(\G): x\in \mathcal O\}>0$. Note that
 $L_\theta(x):=\{(xm_\theta u_t)^+\in \br^2=\partial(\bH^3)-\{\xi\}:t \ge 0\}$ is the line segment connecting
 $x^+$ (at $t=0$) and $x^-$ (at $t=\infty$).
Hence $x\in \mathcal O$ implies that $x^-$ is not a limit point of
the intersection $L_\theta(x)\cap \Lambda(\G)$ for a positive set
of directions $\theta$. This contradicts Theorem \ref{mar1} 
and proves the claim \eqref{cff}.

Let
$\mathcal O$ be an $r$-neighborhood of $\Omega$ for some small $r>0$.
 If $x\in X-\mathcal F$, then $xu_t\in \Omega$
and $xu_{t+s}\in \mathcal O$ for all $|s|<r$. Hence  if
$xu_{t_i}\in \Omega$ for an unbounded sequence $t_i$, 
$\int_{t\in \br}\chi_{\mathcal O} (xu_t)dt =\infty$.
 As $m^{\BR}(\mathcal F)=0$ and $0<m^{\BR} (\mathcal O)<\infty$, this implies the claim by 
Lemma \ref{con}.
\end{proof}

\subsection{Leafwise measures}
Let $W$ be a closed connected subgroup of $N$.
 Let $\mathcal M_\infty(W)$ denote the space of locally finite 
 measures on $W$ with the smallest
 topology so that the map $\nu \mapsto\int \psi \;d\nu$ 
 is continuous for all $\psi\in C_c(W)$ (the weak$^*$ topology).
A locally finite Borel measure $\mu$ on $X$ gives rise to a system of 
locally finite measures
$[\mu_{x}^W]\in \mathcal M_\infty(W)$, unique up to normalization, 
called the leafwise measures or conditional measures on $W$-orbits.
There is no canonical way of normalizing these measure. For our purpose
here, we fix a normalization so that $\mu_x^W(N_1\cap W)=1.$ 
With this normalization, the assignment $x\mapsto \mu_x^W,$ 
is a Borel map, furthermore, for a full measure subset $X'$ of $X$, 
$\mu_{xu}^W=u.\mu_{x}^W$ for every $x, xu\in X'$;
for a comprehensive account on leafwise measures we refer the reader to \cite{EL}.


In the case when $W=N$, we have $\mu_{{H}(x)}^{\PS}=\mu_{x}^{\PS}$ and $\mu_{{H}(x)}^{\Leb}=\lambda_x,$, which are
precisely the $N$-leafwise measures of
$\BMS$ and $\BR$ measure respectively, up to normalization. 
We will be considering the $U$ leafwise measures of $m^{\BMS}$ as well as of $\mu_{E,s}^{\BR}$.

We will use the following simple lemma. 
\begin{lem}\label{minv} Let $\mu$ be a locally finite $M$-invariant measure on $X$.
For any $0<\tau \ll 1$, and any $0\le  \theta<\pi$ we have
$$|\mu_{x m_\theta}^{U^\tau}|=|\mu_{x}^{U_\theta^\tau}| $$
for $\mu$ a.e. $x\in X$.
\end{lem}

\begin{proof} Since $\mu$ is
$M$-invariant, for $T_\rho=A_\rho {\chN_\rho} M_\rho$, we have
$$\frac{|\mu_x^{U_\theta^\tau}|}{|\mu_{xm_\theta}^{U^\tau}|}
=\lim_{\rho\to 0}\frac{\mu(xm_\theta U^\tau m_\theta^{-1} (m_\theta V_\rho T_\rho m_\theta^{-1}))}
{\mu(xm_\theta U^\tau ( V_\rho T_\rho))}
=1 .$$
\end{proof}

\subsection{Recurrence for $m^{\BMS}$}
 Since the frame flow is mixing by Theorem \ref{fm} with respect
to $m^{\BMS}$, we have:

\begin{prop}\label{ae} For any non-trivial $a\in
A$, $m^{\BMS}$ is $a$-ergodic.\end{prop}

\begin{thm}\label{p;bms-rec} \label{c;non-atomic} Let $\delta>1$. For a.e. $x\in \Omega$,
$(m^{\BMS})_x^U$ is atom-free.
\end{thm}
\begin{proof}
Setting $\mathcal F:=\{x\in \Omega: (m^{\BMS})_x^U \text{ has an atom}\}$,
we first claim that $m^{\BMS}(\mathcal F)=0$. Suppose not. Fix any non-trivial $a\in A$. Since $U$ is
normalized  by $a$, $\mathcal F$ is $a$-invariant. Hence $m^{\BMS}(\mathcal F)=1$ by
Proposition \ref{ae}. Using the Poincare recurrence theorem,  it
can be shown that $$\mathcal F':=\{x\in \Omega: (m^{\BMS})_x^U\text{ is the
dirac measure at $e$}\}$$ has a full measure in $\Omega$ (cf. \cite{KS}, \cite[Theorem 7.6]{L}).

Since for any $x\in \Omega$, $\mu_{x}^{\PS}$ is a positive
$\delta$-dimensional Hausdorff measure on 
$\{xn\in {H}(x): (xn)^+\in \Lambda(\G)\}$ by Theorem \ref{psh} and $(m^{\BMS})_x^U=(
\mu_{{H}(x)}^{\PS} )_x^U$ for a.e.  $x\in \Omega$, it follows that for a.e. $x\in
\Omega$, $(\mu_{{H}(x)}^{\PS} )_x^U$ is the Dirac measure at $e$.
By the Fubini theorem, there exists $x\in \Omega$ and a measurable
subset $D_0\subset {H}(x)$ with $\mu_{{H}(x)}^{\PS}(D_0)>0$ such
that for each $y\in D_0$, $(\mu_{{H}(ym_\theta)}^{\PS}
)_{ym_\theta}^U$ is the dirac measure at $ym_\theta$ for a
positive measurable subset of $m_\theta$'s.
For $s\ge 0$, denote by $\mathcal H^s$ the $s$-dimensional
Hausdorff measure on $\Lambda(\G)-\{x^-\}$; so $\mathcal
H^\delta={\mu}_{{H}(x)}^{\PS}$. In the identification of ${H}(x)$
with $\br^2$ via the map $y\mapsto y^+$, this implies that there
is a subset $D_0'\subset \Lambda(\G)-\{x^-\}\subset \br^2$ with
$\mathcal H^{\delta}(D_0')>0$  such that for all $\xi\in D_0'$,
there is  a positive measurable subset of lines $L$ through $\xi$ such that
 $0<\mathcal H^{0}((\Lambda(\G)-\{x^-\})\cap L)<\infty$.
This contradicts Theorem \ref{Mars2} which implies that
$(\Lambda(\G)-\{x^-\})\cap L$ has dimension $\delta-1>0$ for almost all lines $L$ through $\xi$.
\end{proof}


\begin{cor}
If $\delta>1$,  $m^{\BMS}$ is $U$-recurrent, i.e., for any measurable subset $B$ of $X$,  $\{t: xu_t\in B\}$ is
unbounded  for a.e. $x\in B$.
\end{cor}
\begin{proof} By \cite[Theorem 7.6]{EL},
 Theorem \ref{c;non-atomic}  implies that $(m^{\BMS})_x^U$ is infinite
 for a.e. $x$. \cite[Theorem 6.25]{EL} implies the
 claim.
\end{proof}

\subsection{Doubling for the $(\mu^{\BMS})_x^U$.}

As before, we assume $|m^{\BMS}|=1$.
 Since $\Omega$ is a compact subset,
 we have \be\label{rd} \rho:=\tfrac{1}{2}\inf \{\rho_x: x\in \Omega\}>0 .\ee
 
Fix  a small number $\e>0$. It follows from Theorem
\ref{p;bms-rec} that
there exist $0<\beta=\beta(\e)\ll \rho$ and
a compact subset $\Omega_\e '\subset\Omega$ with
$m^{\BMS}(\Omega_\e ')>1-\frac{\e^2}{2}$ such that
\begin{equation}\label{lrho}
(m^{\BMS})_x^U[-3\beta,3\beta]<
\tfrac{1}{2}(m^{\BMS})_x^U[-(\rho-\beta),\rho-\beta]\hspace{5mm}\mbox{for
all} \h\h x\in\Omega'_\e.\end{equation}

Since the covering $\{xB_{ \tau}: x\in \Omega , \tau >0\}$ admits
a disjoint subcovering of $\Omega$ 
with full BMS measure (see \cite[Theorem 2.8]{Ma}), 
there exist $x_0=x_0(\e) \in \Omega_\e'$ and $0<\tau<\beta(\e)$ such that
for $B_{x_0}(\tau):=x_0{\chN}_\tau A_\tau MN_\tau$, 
\be\label{e;l-density3}m^{\BMS}(B_{x_0}(\tau)\cap
\Omega'_\e )>(1-\tfrac{\e^2}{2})\cdot
m^{\BMS}(B_{x_0}(\tau)).\ee

$$\text{We fix $x_0\in \Omega'_\e$ and $ \tau>0$ for the rest of this section.}$$
Recall the notation $T_\tau={\chN}_{\tau}  A_\tau M,$  so that
$B_{x_0}(\tau)=x_0T_\tau N_\tau$.
Set $\nu=\nu_{x_0T_\tau}$ for simplicity.
Using Theorem \ref{p;bms-rec}, we 
will prove: 

\begin{thm}\label{finalo} Let $\delta>1$. Let $c_0> 1$ be as in Theorem \ref{ub} where $F_0$
is the $2\rho$-neighborhood of $\Omega$.
 Then there exists 
a Borel subset $\Xi_{\e}^{\PS}(x_0)\subset x_0T_\tau $
which satisfies the following properties:
\begin{itemize} 
\item[(i)]
$\nu(\Xi_{\e}^{\PS}(x_0))\ge (1-c_0^4\cdot \e)\nu (x_0T_\tau )$; 
\item[(ii)]
for any   $xm_\theta \in \Xi_{\e}^{\PS}(x_0),$ with $\theta$ a $\op{PL}$ direction for $(x,\tau),$
there exists a Borel subset $O_\theta(x)$ of the set
$\{t\in xV_\theta: D(\sigma_{x,\theta}^\tau)( t)>0\}$ such that
$$\mu_{x}^{\PS}( O_\theta(x)  U_{\theta}^{\rho}) \ge 2
\mu_{x}^{\PS}(O_\theta(x)  U_{\theta}^{2\tau}) \ge \tfrac{\tau^\delta}{4c_0}.$$ 
\end{itemize} 
\end{thm}

Despite the rather complicated formulation of this theorem, which is tailored
towards our application later, the theorem is intuitively clear. Indeed $B_{x_0}(\tau)$
is chosen so that for ``most" $\BMS$ points, we have~\eqref{lrho}. On the other hand, 
in view of Corollary~\ref{c;ps-leb}, for a $\op{PL}$ direction $\theta$, the Radon-Nikodym derivative
$D(\sigma_{x,\theta}^\tau)$ is positive $\sigma_{x,\theta}^\tau$-almost everywhere. 
Therefore,  by Fubini's theorem,  for ``most" $\BMS$ points $x\in B_{x_0}(\tau)$, 
``most" points in $xN_\rho$ satisfy both~\eqref{lrho} and the non-vanishing of the Radon-Nikodym
derivative implies Theorem \ref{finalo}.   
The precise treatment of the above sketch of the proof is given in the rest of this subsection.

\begin{lem}\label{psin} Let $xm_\theta \in x_0T_\tau$.
 For any Borel subset $O_\theta'(x)\subset p_\theta(xN_\tau\cap \Omega_\e'm_\theta^{-1})$,
$$\mu_{x}^{\PS}( O_\theta'(x)  U_{\theta}^{\rho}) \ge 2
\mu_{x}^{\PS}(O_\theta'(x)  U_{\theta}^{2\tau}).$$
\end{lem}

 \begin{proof}
If $t\in p_{\theta}({xN_\tau\cap \Omega_\e' m_\theta^{-1}})$
 and
hence $t=xv_\theta$ for $v_\theta\in V_\theta$ where $z_t m_\theta^{-1}=xv_\theta u_\theta$
for $z_t\in \Omega_\e'$,
 we can write it as $tm_\theta =z_tm_\theta^{-1} u_\theta^{-1} m_\theta
 =z_t u_0$ for some $u_0\in U^{\tau}$.
Hence by \eqref{lrho},
$$|\mu_{tm_\theta}^{U^{\rho}}|\ge
|\mu_{z_t}^{U^{\rho-\tau}}|\ge 2|\mu_{z_t}^{U^{3\tau}}| \ge
2|\mu_{tm_\theta}^{U^{2\tau}}| .$$

 Using this and since $O_\theta'(x)\subset p_\theta(xN_\tau\cap \Omega_\e'm_\theta^{-1})$, we have
\begin{align*} & \mu_{x}^{\PS}( O_\theta'(x)
U_{\theta}^{\rho} ) = \int_{ O_\theta'(x)}
|\mu_t^{U_{\theta}^{\rho}} |d\sigma_{x,\theta}^{\rho} (t)
=\int_{ O_\theta'(x) } |\mu_{tm_\theta}^{U^{\rho}}| d\sigma_{x,\theta}^{\rho}
(t)
 \\&\ge 2 \int_{ O_\theta'(x) }
| \mu_{tm_\theta}^{U^{2\tau}} |d\sigma_{x,\theta}^{\rho}
(t)
=2 \int_{ O_\theta'(x) }
|\mu_t^{U_{\theta}^{2\tau}}| d\sigma_{x,\theta}^{\rho} (t)
 =2 \mu_{x}^{\PS}(O_\theta'(x) U_{\theta}^{2\tau}).
\end{align*}
\end{proof}

\begin{lem}\label{oe}
There exists a compact subset $\Omega_\e\subset \Omega_\e'$  with
$m^{\BMS}(\Omega_\e )>1-{\e^2}$ such that
$$m^{\BMS}(B_{x_0}(\tau)\cap
\Omega_\e )>(1-{\e^2})\cdot m^{\BMS}(B_{x_0}(\tau))$$
and that
 $$ p_\theta(xN_\tau\cap \Omega_\e m_\theta^{-1})\subset \{t\in xV_\theta: D(\sigma_{x,\theta}^\tau)( t)>0 \}$$
for all $xm_\theta\in x_0T_\tau$ with
$\theta$ a $\op{PL}$-direction for $(x,\tau)$.
\end{lem}
\begin{proof} Setting $xV_\theta':=\{t\in xV_\theta : D(\sigma_{x,\theta}^\tau, t)
=0\}$,
we have
\begin{align*} &m^{\BMS}( \cup_{xm_\theta \in x_0T_\tau}
xV_\theta' U_\theta^\tau) \\ & =
\int_{xm_\theta \in x_0T_\tau} \mu_{x}^{\PS}(xV_\theta'U_\theta^\tau)d\nu(xm_\theta)
\\&= \int_{xm_\theta \in x_0T_\tau} \int_{t\in xV_\theta'} |(\mu_{x}^{\PS})_x^{U_\theta^\tau}|
d\sigma_{x,\theta}^\tau(t)d\nu(xm_\theta) =0 .
\end{align*}
Hence there exists an open subset $\mathcal O_\e$ of
$B_{x_0}(\tau)$ which contains the subset $\cup_{xm_\theta \in
x_0T_\tau} xV_\theta' U_\theta^\tau$ and
$m^{\BMS}(\mathcal O_\e)\le \frac{\e ^2}2 \cdot
m^{\BMS}(B_{x_0}(\tau))$. Now set $\Omega_\e:=\Omega'_\e -
\mathcal O_\e$. It is easy to check that this $\Omega_\e$
satisfies the claim.
\end{proof}
We set $$\Xi_{\e}^{\PS}(x_0):= \{x\in x_0T_\tau : \mu_x^{\PS}(x N_\tau \cap \Omega_\e)>
(1-\e) \mu_{x}^{\PS}(x N_\tau )
  \}.$$

\begin{lem}\label{l;fubini}
We have
 $$\nu(\Xi_{\e}^{\PS}(x_0))\ge (1-c_0^{4}\cdot \e)\nu (x_0
T_\tau ).$$
\end{lem}
\begin{proof}

Set $\te b_1:=\inf_{x\in x_0T_\tau} \tfrac{\mu_{x_0}^{\PS}(x_0N_\tau)} {\mu_x^{\PS}(xN_\tau)}$ and
$\te b_2:=\sup_{x\in x_0 T_\tau} \tfrac{\mu_{x_0}^{\PS}(x_0N_\tau)}{\mu_x^{\PS}(xN_\tau)}$.
Since $x_0^+\in \Lambda(\G)$ and $x_0^+=x^+$, we have $x^+\in \Lambda(\G)$
and hence it follows from Theorem \ref{ub} that $c_0^{-4}<\tfrac{b_2}{b_1}<c_0^4$.

Note that
\begin{align*} &m^{\BMS}(B_{x_0}(\tau) \cap \Omega_\e)\\&
\le \int_{\Xi_{\e}^{\PS}(x_0)} \mu_y^{\PS}(yN_\tau)d \nu(y)
+ (1- \e) \int_{x_0 T_\tau - \Xi_{\e}^{\PS}(x_0)} \mu_y^{\PS}(yN_\tau)d \nu(y)\\
&\le m^{\BMS}(B_{x_0}(\tau)) - \tfrac{\e}{ b_2}   \nu(x_0 T_\tau - \Xi_{\e}^{\PS}(x_0)) \cdot \mu_{x_0}^{\PS}(x_0N_\tau) .
\end{align*}

By Lemma \ref{oe}, it follows that
$$\e^2 \cdot m^{\BMS}(B_{x_0}(\tau))\ge   \tfrac{\e}{ b_2}\nu(x_0 T_\tau - \Xi_{\e}^{\PS}(x_0))\cdot \mu_{x_0}^{\PS}(x_0N_\tau)$$ and hence
$$
{\e} \cdot b_1^{-1}  \nu (x_0 T_\tau
)\cdot \mu_{x_0}^{\PS}(x_0N_\tau) \ge  \tfrac{1}{ b_2} \nu(x_0 T_\tau - \Xi_{\e}^{\PS}(x_0))\cdot \mu_{x_0}^{\PS}(x_0N_\tau).$$ Therefore
$$
 \nu(x_0 T_\tau - \Xi_{\e}^{\PS}(x_0)) \le
{\e} \cdot \tfrac{b_2}{ b_1} \nu (x_0 T_\tau
) ,$$ implying the claim.
\end{proof}

By the $M$-invariance of $m^{\BMS}$, by Lemma \ref{minv},
$$\mu_{xm_\theta}^{\PS}(\Omega_\e\cap xm_\theta N_\tau)=
\mu_{xm_\theta }^{\PS}(\Omega_\e \cap xN_\tau m_\theta)
=\mu_{x}^{\PS}(\Omega_\e m_\theta^{-1} \cap xN_\tau )
.$$
Note that  for any $xm_\theta \in \Xi_{\e}^{\PS}(x_0)$,
we have $$\mu_{x}^{\PS}( xN_\tau \cap \Omega_\e
m_\theta^{-1})>(1- \e)\mu_x(xN_\tau) \ge\tfrac{\tau^\delta}{2c_0} .$$

By setting $O_\theta(x):=p_\theta(xN_\rho\cap \Omega_\e m_\theta^{-1})$,
we note that $O_\theta(x)  U_{\theta}^{2\tau}$ contains
$x N_\tau \cap \Omega_\e m_\theta^{-1}$ and hence for all small $0<\e\ll 1$,
$$\mu_x^{\PS}(O_\theta(x)U_\theta^{2\tau})\ge  (1-\e) \mu_x^{\PS}(x N_\tau \cap \Omega_\e m_\theta^{-1})
\ge \tfrac{\tau^\delta}{2c_0} .$$

 Therefore the above three lemmas prove Theorem \ref{finalo}.

\section{Window theorem for Hopf average}\label{sec;window}
We will combine the results from previous sections and prove the window theorem
\ref{wtpf} in this section. 
We first show that the disintegration along $U$ of $\lambda_s,$
notation as in Section~\ref{sec;cond}, 
has certain doubling properties, see Theorem~\ref{deduce}. 
This is done by applying results in Section~\ref{sec;energy}, 
in particular the key lemma, to $\lambda_{E,x,s}$ 
and the limiting measure $\mu_x^{\PS},$ in combination
with Theorem~\ref{finalo}, which gives 
a rather strong doubling property
for the disintegration of the $\PS$ measure. 
As we mentioned in the introduction, in general, 
the weak* convergence of measures does not give control
on the corresponding conditional measures,  
e.g., one should recall the well-known discontinuity of the entropy. 
However, here the key lemma gives a good control 
both on the prelimiting measures $\lambda_{E,x,s}$ 
and the limit measure $\mu_x^{\PS},$ 
and helps us to draw some connection between the conditionals. 

In order to obtain the Window Theorem~\ref{p;windowone},
we flow by $a_{-s}$ for a suitably big $s$ 
and bring $[-T,T]$ to size $[-\rho,\rho].$ 
We  are now working with $m^{\BR}_{E,s}$
rather than $\tfrac{1}{m^{\BR}( E)}m^{\BR}|_E,$
and the desired estimate follows from Theorem~\ref{deduce}

\subsection{Window theorem for $\chi_E$} 
Let $\G$ be a convex cocompact subgroup with $\delta>1$.
Let $E$ be a BMS box. 
For simplicity, we set
 $$\mu_s:=\mu_{E,s}^{\BR}\quad\text{and}\quad \lambda_{x,s}=\lambda_{E,x, s}$$ defined in section 5.
For $0<r\le 1$ and $\rho>0$ as in \eqref{rd}, we put
$$ E_{s}(r):=\{x\in E a_{-s}:
\frac{(\mu_{s})_x^U[-2 \rho r, 2\rho
r]}{(\mu_{s})_x^U[-2 \rho, 2 \rho]}>1-r\}.$$

\begin{thm}\label{deduce}  
There exist $0<r_0<1$ and $s_0>1$ (depending on $E$) such that for all $s>s_0$, we
have
 $$\mu_{E,s}^{\BR}(E_s(r_0))<1-r_0 .$$
\end{thm}

\begin{proof}
 Suppose not; then there
is a subsequence $r_i\rightarrow 0$ and a subsequence
$s_i\rightarrow +\infty$ such that $\mu_{E,s_i}^{\BR}(E_{s_i}(r_i))\geq1-r_i$.
Set $$\mu_i=\mu_{E,s_i}^{\BR}\quad\text{ and }\quad E_i:=E_{s_i}(r_i) .$$

Fix $\e>0$. Let $x_0\in \Omega_\e'$, $0<\tau<\rho$, $c_0>1$ and $\Xi^{\PS}_\e(x_0)$ be as in Theorem \ref{finalo}. 
 Set $q_0:=m^{\BMS}(B_{x_0}(\tau))>0$
and
\[x_0 T_i:=\{y \in x_0T_\tau: \lambda_{s_i,y}(E_i\cap yN_\tau)>
\left(1-\sqrt{\tfrac{2r_i}{q_0}}\right) \lambda_{s_i,y}(yN_\tau)\}.\] 

Recall the measure $\nu=\nu_{x_0T_\rho}$ from Theorem~\ref{finalo}.
We claim that for all large $i\gg 1$, we have 
\begin{equation} \label{vx_0T_i}
\nu(x_0T_i) \ge (1-4\sqrt{r_i})\nu(x_0T_\tau).
\end{equation}
We will first show that for all large $i\gg 1$,
\be\label{mushow} \mu_i(E_i\cap B_{x_0}(\tau)) \ge(1-\tfrac{2r_i}{q_0}) \mu_i( B_{x_0}(\tau)).\ee

If this does
not hold, by passing to a subsequence, we have that
\begin{multline*} 
1-r_i<\mu_i(E_i)=\mu_i(E_i\cap B_{x_0}(\tau))
+\mu_i(E_i - B_{x_0}(\tau)) \\
\le |\mu_i| -\tfrac{2r_i}{q_0} \mu_i( B_{x_0}(\tau)).
\end{multline*}
Since $|\mu_i|=1$, it follows that
$$\frac{2}{q_0} \mu_i( B_{x_0}(\tau))\le 1.$$
On the other hand, since $\mu_i$ weakly converges to $m^{\BMS}$ by Theorem \ref{wc1},
we have $$\frac{\mu_i(B_{x_0}(\tau))}{q_0}\to 1,$$ which gives a
contradiction. This shows \eqref{mushow}.
Now,
by the same type of argument as the proof of Lemma
\ref{l;fubini}, we can show \eqref{mushow} implies \eqref{vx_0T_i}.

Passing to a subsequence, which we continue to denote by $r_i,$ we
assume that $4 \sum_i \sqrt r_i<\e/2$ and $2r_i/q_0<\e$ for all $i$. If we set $$\Xi^* (x_0) :=\cap_i x_0T_i ,$$ then
it follows that $$\nu(\Xi^*(x_0))>(1-\e)\nu(x_0 T_\tau).$$

Hence for all sufficiently small $\e>0$,
\[\nu( \Xi_\e^{\PS}(x_0)\cap \Xi^*(x_0) )> (1- (1+c_0^4)  \e) \hh\nu(x_0 T_\tau)>0.\]

 Let  $\Theta_\e(x)$ be given as in the Key Lemma~\ref{c;unif-l2}
for $\{\rho, \tau\}$ applied to the set $E$ and the sequence $s_i$.
Since $\op{supp}(\nu)\subset \{x\in X: x^-\in \Lambda(\G)\}$,
 we can find $xm_\theta \in \Xi^*(x_0) \cap \Xi_\e^{\PS}(x_0)$ 
(depending on $\e>0$) with $(xm_\theta)^-=x^-\in \Lambda(\G)$ and $\theta\in \Theta_\e(x)$.
Since $x^+=x_0^+$, we have $x^{\pm}\in \Lambda(\G)$.

By the $M$-invariance of the measure
$\mu_{s_i}$ and as $xm_\theta\in x_0T_i$, we have
\begin{align*}  &\lambda_{x,s_i}(E_i m_\theta^{-1} \cap x N_\tau)
\\&=
\lambda_{xm_\theta ,s_i}(E_i \cap x m_\theta N_\tau)\\
&\ge (1-\sqrt{\tfrac{2r_i}{q_0}}) \lambda_{xm_\theta ,s_i}( x m_\theta N_\tau)\\
&= (1-\sqrt{\tfrac{2r_i}{q_0}}) \lambda_{x ,s_i}( xN_\tau) ,\notag
\end{align*}
and hence 
\be \label{zero1} \lambda_{x,s_i}(xN_\tau- E_i m_\theta^{-1} )\to 0.\ee

Let $\{s_{j_i}\}$ be the
corresponding subsequence given by Lemma~\ref{c;unif-l2} depending 
on $(x, \theta)$. By
passing to that subsequence, we set $s_i:=s_{j_i}$.

For $O_\theta(x)$ as in Theorem \ref{finalo}, 
 we consider $ L_i:=p_\theta(O_\theta(x) U_\theta^\tau
 \cap E_i m_\theta^{-1}) .$  By Lemma \ref{fff3} below,
\begin{equation*}\lambda_{x,s_i}(L_i U_\theta^{\rho})
\le  \tfrac{1}{1-r_i} \lambda_{x,s_i}( L_i U_\theta^{2\tau}) .\end{equation*}

 Since $L_i=p_\theta(E_i m_\theta^{-1})\cap O_\theta(x)$
and
 $\{t\in O_\theta(x): D(\sigma_{x,\theta}^\tau, t)=0\}=\emptyset$, 
 it follows from the key Lemma \ref{zerotwo} and \eqref{zero1}
that for
all large $i\gg 1$,
\begin{align}\label{e11} 
& \lambda_{x,s_i}(O_\theta(x) U_\theta^{\rho} ) =\sigma_{x,\theta,s_i}^{\rho}(O_\theta(x))
\\ & \leq
(1+\e)\sigma_{x,\theta,s_i}^{\rho}(L_i) 
 =(1+\e)\lambda_{x,s_i}(xL_iU_\theta^{\rho}) .\end{align}

Therefore we have
\begin{align*} \lambda_{x, s_i}(O_\theta(x) U_\theta^{\rho} )
& \le \tfrac{(1+\e)}{(1-r_i)} \lambda_{x, s_i}
(O_\theta(x) U_\theta^{2\tau})   \\ & \le
(1+2\e) \lambda_{x, s_i}
(O_\theta(x) U_\theta^{2\tau}).
\end{align*} 

Recall we chose $s_i=s_{j_i}$ so that Theorem~\ref{c;unif-l2}
holds for $\{\rho,\tau\}.$
Hence by sending $s_i\to \infty$, the above implies
$$\mu_{x}^{\PS}(O_\theta(x) U_\theta^{\rho} )\le (1+2\e)
\mu_{x}^{\PS}(O_\theta(x) U_\theta^{2\tau} ).$$
Together with Theorem \ref{finalo},
this gives
 $$2\e \cdot
\mu_{x}^{\PS}(O_\theta(x) U_\theta^{2\tau} )\ge 
\mu_{x}^{\PS}(O_\theta(x) U_\theta^{2\tau} ),$$
however, $\mu_{x}^{\PS}(O_\theta(x) U_\theta^{2\tau} )\ge \tfrac{\tau^\delta }{8c_0}>0.$ This gives a contradiction and finishes the proof of Theorem \ref{deduce}.
\end{proof}

\begin{lem}\label{fff3} 
Let $xm_\theta \in x_0T_\tau$ for $x_0\in \Omega$ and $s>0$. 
For any Borel subset $L_\theta(x)\subset p_\theta(xN_\tau \cap E_s(r) m_\theta^{-1})$ and for all sufficiently small $0<r\ll 1$,
\be\label{e;window0}\lambda_{x,s}(L_\theta(x) U_\theta^{\rho})
\le \tfrac{1}{1- r} \lambda_{x,s}( L_\theta(x)  U_\theta^{2\tau}) .\ee
\end{lem}

\begin{proof} 
Set $$\tilde L=p_\theta(xN_\tau \cap E_s(r) m_\theta^{-1}) .$$
 Note that if $t\in \tilde L$ and hence $t\in p_\theta(E_s(r)
m_\theta^{-1})$, then $tm_\theta=z_tu_0$ for some $z_t\in E_s(r)$ and
$u_0\in U^{\tau}$. On the other hand, it follows from the definition of $E_s(r)$ that $$
|(\mu_{s})_{ z_t}^{U^{2\rho}}|\\
 \le  \tfrac{1}{1- r} |(\mu_{s})_{z_t}^{U^{2r\rho}}|$$
and hence
\begin{multline*} |(\mu_s)_{tm_\theta}^{U^{\rho}}| \le |(\mu_s)_ {z_t}^{U^{\rho+\tau}}| \le
|(\mu_s)_{ z_t}^{U^{2\rho}}|
 \le  \tfrac{1}{1- r} |(\mu_s)_{z_t}^{U^{2r\rho}}| \le   \tfrac{1}{1- r}|(\mu_{s})_{tm_\theta}^{U^{2r \rho+\tau}}| .\end{multline*}  
 Therefore, since $\mu_{E,s}^{\BR}$ is $M$-invariant, by Lemma \ref{minv},
\begin{align*}
\lambda_{ x,s}(L_\theta(x) U_\theta^{\rho}) &=
 \int_{t\in L_\theta(x)}|(\mu_{s})_t^{U_{\theta}^{\rho}}|\;
d{\sigma_{s,x,\theta}^{\rho} }(t)\\& =\int_{t \in L_\theta(x)  }|(\mu_s)_{t m_\theta }^{U^{\rho}}|
\; d{\sigma_{x,\theta,s}^{\rho} }(t)\\
&\le \tfrac{1}{1- r}\int_{t\in
L_\theta(x)}|(\mu_s)_{tm_\theta}^{U^{2\tau}}| \; d{\sigma_{x,\theta, s}^{\rho} }(t)
\\ & = \tfrac{1}{1- r}\int_{t\in
L_\theta(x)}|(\mu_s)_{t}^{U_\theta^{2\tau}} |\; d{\sigma_{x,\theta,s}^{\rho} }(t)
\\ &=
 \tfrac{1}{1- r} \lambda_{x, s}( L_\theta(x) U_{\theta}^{2\tau}).
\end{align*}
\end{proof}

We deduce the following from Theorem \ref{deduce}:
\begin{thm}\label{p;windowone}  
There exist $0<r<1$ and $T_0>1$, depending on $E$, 
such that for all  $T>T_0$ 
\[m^{\BR} \{x\in E: \int_{-rT}^{rT}\chi_E(xu_t)dt <(1-r)\cdot 
 \int_{-T}^{T}\chi_E(xu_t)dt\} \ge r\cdot  m^{\BR}(E) .\]

\end{thm}\begin{proof}
Setting
$$
E(s,r):=\{x\in E: \tfrac{\int_{-2r\rho e^s}^{2r
\rho e^s }\chi_E(xu_t)dt}{\int_{-2\rho e^s}^{2 \rho
e^s}\chi_E(xu_t)dt}\geq 1-r \},
$$ 
it suffices to prove that for some $0<r<1$ and for all $s$ large.
\be \label{mbr2} 
m^{\BR}( E(s,r))<(1-r)m^{\BR}(E).
\ee

We note that 
$$
E(s,r)=\{x\in E: \tfrac{(\mu^{\BR}_{0,E})_x^U[-2r
\rho e^s ,2r \rho e^s ]} {(\mu^{\BR}_{0,E})_x^U[-2\rho e^s,
2\rho e^s]}\geq1-r\},
$$ 
where $(\mu^{\BR}_{E,0})_x^U$ denotes the
leafwise measure of $\mu_{E,0}^{\BR}=\frac{m^{\BR}|_E}{m^{\BR}(E)}$.

Note that \eqref{mbr2} follows from Theorem~\ref{deduce} if we show 
\begin{equation*}\label{compare} 
m^{\BR}(E(s,r))= m^{\BR}(E) \cdot \mu_{E,s}^{\BR}(E_s( r))\;\mbox{ for all } 0<r<1.
\end{equation*}
  
We now show the above identity. Let $s$ be fixed.
Then for $\BR$ a.e. points $x$, we have
\begin{align*}
\frac{(\mu^{\BR}_{E,0})_x^U[-2r\rho e^s ,2r \rho e^s ]}
{(\mu^{\BR}_{E,0})_x^U[-2\rho e^s ,2 \rho e^s ]} &=
\frac{\int_{-2r\rho}^{2r\rho}\chi_E(xa_{-s}u_t a_s) dt} 
{\int_{-2\rho}^{2\rho}\chi_E(xa_{-s}u_t a_s) dt}
&&\mbox{ since $m^{\BR}$ is $U$-invariant}\\
&=\frac{\int_{-2r\rho}^{2r\rho}\chi_{Ea_{-s}}(yu_t ) dt}
{\int_{-2\rho}^{2\rho}\chi_{Ea_{-s}}(yu_t ) dt}
&&\mbox{ $y=xa_{-s}$ }\\
&=\frac{(\mu^{\BR}_{Ea_{-s},0})_y^U[-2r\rho ,2r \rho ]}
{(\mu^{\BR}_{Ea_{-s},0})_y^U[-2\rho ,2 \rho ]} 
&&\mbox{ since $m^{\BR}$ is $U$-invariant}\\
&=\frac{(e^{(\delta-2)s}\mu^{\BR}_{E,s})_y^U[-2r\rho,2r \rho]}
{(e^{(\delta-2)s}\mu^{\BR}_{E,s})_y^U[-2\rho,2 \rho]} 
&&\mbox{ by the definition of $m^{\BR}_{E,s}$}.
\end{align*}
Hence $a_{-s}E(s,r)$ coincides with $E_s( r),$
up to a BR null set. This implies the
claim using the definition of $\mu_{E,s}^{\BR}$.
\end{proof}

\subsection{Ergodic decomposition and the Hopf ratio theorem}\label{erg}
In this subsection, let $\mu$ be a locally finite $U$-invariant conservative measure on $X$.
Let $\mathcal M_\infty(X)$ denote the space of locally finite measures on $X$ with weak$^*$ topology.

Let $\mathcal A$ denote a countably generated $\sigma$-algebra equivalent to the $\sigma$-algebra
of all $U$-invariant subsets of $X$.
There exist a $\mathcal A$-measurable conull set $X'$ of $X,$ a family $\{\mu_x=\mu_x^{\mathcal A}:x\in X'\}$
of conditional measures on $X$ and a probability measure ${\mu}_*$ on $X$ which give rise to the ergodic decomposition of $\mu$:
$$\mu=\int\mu_x\; d{\mu}_*(x)$$
where
the map $X'\to \mathcal M_\infty(X)$, $x\mapsto \mu_x$, is Borel measurable,
  $\mu_x$ is a $U$-invariant, ergodic and conservative measure on $X$
and for any $\psi\in L^1(X, \mu)$,
 $$\mu(\psi)=\int_{x\in X}\mu_x(\psi) \; d{\mu}_*(x);$$
see \cite[5.1.4]{EL}. 

The following is the 
 Hopf ratio theorem in a form  convenient for us (\cite{Hopf}, see also \cite{Zw}).
\begin{thm}\label{hopf} 
Let $\psi, \phi\in L^1(\mBR)$ with $\phi\ge 0$. Furthermore suppose that $\psi$ and $\phi$
are compactly supported. Then 
\begin{equation*}\label{c1}\lim_T \frac{\int_0^T \psi(xu_t)dt}{\int_0^T  \phi (xu_t)dt}
= \lim_T \frac{\int_{-T}^0 \psi(xu_t)dt}{\int_{-T}^0  \phi (xu_t)dt}
=\frac{\mu_x(\psi)}{\mu_x(\phi )}\end{equation*}
for $\mu$-a.e.
$x\in \{x\in X: \sup_T \int_0^T \phi(xu_t) dt >0\}$.\footnote{note that for a.e. $x$ in this set $\mu_x(\phi)>0,$ see~\cite[Page 3]{Zw}.}
\end{thm}


\begin{lem}\label{lem;unif-conv}\label{ergs}\label{cor;unif-unif} 
Fix a compact subset $E\subset X $ with $\mu(E)>0$. Let
 $\phi$ be a non-negative compactly supported Borel function on $X$ such that $\phi|_{E}>0$.
For any $\rho>0$ there exists a compact subset $E_\rho(\phi)\subset E$ with $\mu(E-E_\rho(\phi))<\rho\cdot
 \mu(E)$ satisfying:
\begin{enumerate}
 \item  the map $x\mapsto \mu_x$ is continuous for all $x\in E_\rho(\phi)$;
\item $\inf_{x\in E_\rho}\mu_x(\phi) >0$;
\item  for any $\psi\in C_c(X)$ the convergence
 \[\frac{\int_0^T\psi(xu_t)dt}{\int_0^T\phi (xu_t)dt}\to \frac{\mu_x(\psi)}{\mu_x(\phi )}\]
 is uniform on $E_\rho(\phi)$.
\end{enumerate}
\end{lem}

\begin{proof} 
By Lusin's theorem, there exists a compact subset $E'\subset E$ with
$\mu(E-E')<\frac{\rho}3 \mu(E)$ and
the map $x\mapsto \mu_x$ is continuous on $E'.$
Since $\int_0^T\phi(xu_t)dt\to +\infty$ for a.e. $x\in E$ by the conservativity of $\mu$,
we have $\mu_x(\phi)>0$ almost all $x\in E$.
Since $x\mapsto \mu_x(\phi)$ is a measurable map, it follows again by Lusin's theorem that
there exists a compact subset $E''\subset E'$ with
$\mu(E'-E'')<\frac{\rho}3 \mu(E')$, $\int_0^\infty \phi(xu_t)dt=\infty$ for all $x\in E''$, and
 $\inf_{x\in E''}\mu_x(\phi) >0$.

We claim that for any $\e>0$ and any compact subset $Q$ of $X$,
there exists a compact subset $E_0 =E_0(Q, \e) \subset E''$ such that 
$\mu(E''-E_0)<\e  \mu(E'')$ and
for all $\psi\in C(Q)$, the convergence
\[\frac{\int_0^T\psi(xu_t)dt}{\int_0^T\phi (xu_t)dt}\to \frac{\mu_x(\psi)}{\mu_x(\phi )}\]
 is uniform on $E_0$.
Let $\mathcal B=\{\psi_j\}$ be a countable dense subset of $C(Q)$ which includes the constant function $\chi_Q.$ 
We can deduce from the Hopf ratio Theorem \ref{hopf} and Egorov's theorem that
 there is a compact subset $E_1\subset E''$ such that  $\mu(E''-E_1)<\frac{\e}{3}  \mu(E'')$,
$\sup_{x\in E_1}\mu_x(Q)<\infty$, and for each $\psi_j\in \mathcal B$, the convergence 
\be\label{e;psi-i}\frac{\int_{0}^T \psi_j(xu_t)dt}{\int_{0}^T \phi (xu_t)dt}\to \frac{\mu_x(\psi_j)}{\mu_x(\phi )}\ee
is uniform on $E_1$.  We will show the uniform convergence in $E_1$
 for all $\psi\in C(Q).$ 
For any $\eta>0$, there exists $\psi_j\in \mathcal B$ 
such that $\|\psi_{j}-\psi\|_\infty<{\eta}.$ 
Let $T_0\gg 1$ be such that
\[\left|\frac{\int_0^T\psi_j(xu_t)dt}{\int_0^T\phi (xu_t)dt}- \frac{\mu_x(\psi_j)}{\mu_x(\phi )}\right| \le {\eta}, \quad
\left|\frac{\int_0^T\chi_Q(xu_t)dt}{\int_0^T\phi (xu_t)dt}- \frac{\mu_x(\chi_Q)}{\mu_x(\phi )}\right| \le {\eta} \]
for all $x\in E_1$ and $T\ge T_0$.
Now for any $x\in E_1$ and $T\ge T_0$, we have
\begin{align*} &\left|\frac{\int_{0}^T\psi(xu_t)dt}{\int_{0}^T \phi (xu_t)dt}-
\frac{\mu_x(\psi)}{\mu_x(\phi )}\right|\\
&\leq \frac{\int_{0}^T|\psi(xu_t)-\psi_{j_0}(xu_t)|dt}{\int_{0}^T \phi (xu_t
)dt}+ \left|\frac{\int_{0}^T\psi_{j_0}(xu_t)dt}{\int_{0}^T \phi (xu_t)dt}-\frac{\mu_x(\psi_{j_0})}{\mu_x(\phi )}\right|  +\left|\frac{\mu_x(\psi_{j_0})}{\mu_x(\phi )}-\frac{\mu_x(\psi)}{\mu_x(\phi )}\right|\\
&\le  
\frac{\int_{0}^T \chi_Q(xu_t) dt}{\int_{0}^T \phi (xu_t
)dt}\|\psi-\psi_j\|_\infty  + {\eta} + \frac{\mu_x(Q)}{\mu_x(\phi)}\|\psi-\psi_j\|_\infty \\
&\leq \frac{\mu_x(Q)}{\mu_x(\phi )}\eta  + \eta^2  +\eta+\frac{\mu_x(Q)}{\mu_x(\phi)} \eta\\
&\le \eta( 2 a_0+\eta+1)
\end{align*}  where $a_0:=\sup_{x\in E_1} \frac{\mu_x(Q)}{\mu_x(\phi)}<\infty$.
This proves the claim.
Let $Q_1\subset Q_2\subset\cdots$ be an exhaustion of $X$ by  compact sets.
 Then $E_\rho(\phi):=\cap_i E_0(Q_i, \frac{\rho}{4^{i+1}}) $ satisfies all the desired properties.
\end{proof}

\subsection{Window theorem for $\psi\in C_c(X)$ with $\psi|_E>0$}\label{brerg}

Let $\G$ be a convex cocompact subgroup with $\delta>1$.
Let $\mathcal A$ denote a countably generated $\sigma$-algebra which is equivalent to
the $\sigma$-algebra of all $U$-invariant subsets of $X,$ as before.

Since $m^{\BR}$ is $U$-conservative by Theorem \ref{conser},  we may write an ergodic decomposition
$$m^{\BR}=\int_{x\in X'} \mu_x \; d m^{\BR}_*(x)$$
where $X'$ is a $\mathcal A$-measurable conull set of $X$, $m^{\BR}_*$ is a probability measure on $X$, and 
for all $x\in X'$, $\mu_x=\mu_x^{\mathcal A}$ is a $U$-invariant ergodic conservative measure.

\begin{lem}\label{sone} 
Let $E$ and $0<r<1$ be as in Theorem \ref{p;windowone}.
 Let $\psi\in C_c(X)$ with $\psi|_E>0$. 
 For any $\rho>0$, there exists $s_0\ge 1$ such that for all $s>s_0$,
$$m^{\BR}\{x\in E: \int_{-rs}^{rs} \psi (x u_t)dt \le (1-r+\rho)\int_{-s} ^s\psi
(xu_t) dt \}\ge  (r -2\rho) \cdot  m^{\BR}(E).$$ 
\end{lem}
\begin{proof}
 For simplicity, set
 $$F_\rho(s):=\{x\in E: \int_{-rs}^{rs} \psi (x u_t)dt \le (1-r+\rho)\int_{-s} ^s\psi
(xu_t) dt \}.$$ Let $E_\rho(\chi_E)\subset E$ be as in Lemma~\ref{lem;unif-conv}.
Since $\psi|_E>0$, there is a subset $E'_\rho$ of $E_\rho(\chi_E)$ such that $\mBR(E-E'_\rho)<2\rho \cdot \mBR(E)$ and
 $\inf_{x\in E'_\rho} \frac{\mu_x(\psi)}{\mu_x(E)}>0$.
Then for all large $s$
(uniformly for all $x\in E_\rho(\chi_E)$),
$$\int_{-s}^{s} \psi(x u_t) dt =\left(\tfrac{\mu_x(\psi)}{\mu_x(E)} +a_x(\psi,s)\right) \int_{-s}^{s} \chi_E(xu_t)
dt $$ where $|a_x(\psi,s)|\le a(s)\to 0$ as $s\to \infty$ by Lemma \ref{ergs}.

Setting $$\tilde E(s,r)=\{x\in E: \int_{-rs}^{rs}\chi_E(xu_t)dt\geq(1-r)\cdot 
 \int_{-s}^{s}\chi_E(xu_t)dt\},$$ 
we claim that $$E'_\rho\cap (E-\tilde E(s,r))\subset F_\rho(s)\quad\text{ for all large $s$},$$
from which the lemma follows by Theorem \ref{p;windowone}.
 For any $x\in E'_\rho\cap (E-\tilde E(s,r))$,  
\begin{align*}&
\int_{-rs}^{rs} \psi(x u_t) dt =(\tfrac{\mu_x(\psi)}{\mu_x(E)}+a_x(\psi, rs)) \int_{-rs}^{rs} \chi_E(xu_t)
dt \\ &\le
(\tfrac{\mu_x(\psi)}{\mu_x(E)} +|a_x(\psi, rs)|)(1-r) \int_{-s}^{s} \chi_E(xu_t)
dt\\ &\leq
(1-r) \int_{-s}^{s} \psi (xu_t)
dt + (|a_x(\psi, s)| +|a_x(\psi,rs)|)( 1-r) \int_{-s}^{s} \chi_E(xu_t).
\end{align*}

Let $s_1>1$ be such that for $s\ge s_1$ and for all $x\in E'_\rho$,
$$\frac{ (|a_x(\psi, s)| +|a_x(\psi,rs)|)( 1-r)}{|\frac{\mu_x(\psi)}{\mu_x(E)}
+a_x(\psi,s)|} \le \rho;$$ this is possible since $\frac{\mu_x(\psi)}{\mu_x(E)}$ is
uniformly bounded from below by a positive number. Then  the claim holds.
\end{proof}

By taking $\rho=r/4$ and replacing $3r/4$ by $r$ in the above lemma, we now obtain:
\begin{thm}[Window Theorem]\label{wtpf} 
Let $\psi \in C_c(X)$ be a non-negative function such that $\psi|_{E}>0$.
Then there exist $0<r <1$ and $T_0> 1$ such that for any $T\ge T_0$,
\begin{equation*}
m^{\BR} \{x\in E: \int_{-rT}^{rT} \psi(xu_t) dt < (1-{r}) \int_{-T}^{T} \psi(xu_t) dt\}> \tfrac{r}{2}\cdot  m^{\BR}(E) .
\end{equation*}
\end{thm}
It is worth mentioning that $r$ obtained here may be rather small.  
The following lemma demonstrates how the window estimates for a sequence will be used.
\begin{lem}\label{ew} Let $\e>0$ and
 a sequence $s_k\to +\infty$ be given. Let $E$ and $\psi$ be as in Theorem \ref{wt}. Fix $\rho>0$.
Let $x_k\in E_\rho(\psi)$ be a sequence 
satisfying $$ \int_{(1-\e)s_k}^{s_k} \psi(x_ku_t) dt \ge c\;\e  \int_{0}^{s_k} \psi(x_k u_t) dt$$
 for some $c>0$ independent of $k$.
Then for any $f\in C_c(X)$, as $k\to \infty$,
$$\frac{\int_{(1-\e)s_k}^{s_k}
 f (x_k u_t)dt}{\int_{(1-\e)s_k}^{s_k} \psi
(x_k u_t)dt} \sim \frac{\mu_{x_k}(
f)}{\mu_{x_k}(\psi)}.$$
\end{lem}
\begin{proof}
By the Hopf ratio theorem, and Lemma \ref{ergs}, we have
$$
\int_0^{s} f(x_k u_t) dt = \tfrac{\mu_{x_k}(f)}{\mu_{x_k}(\psi)}
\int_0^{s} \psi(x_k u_t) dt + a_{x_k}(s)\int_0^{s} \psi(x_k u_t) dt
$$
with $\lim_{s\to\infty}a_{x_k}(s)=0,$ uniformly in $\{x_k\}$.
Therefore
  \begin{multline*} \int_{(1-\e)s_k}^{
s_k} f (x_k
u_t)dt=\tfrac{\mu_{x_k}(f)}{\mu_{x_k}(\psi)}\int_{(1-\e)s_k}^{
s_k} \psi (x_k u_t)dt\\ +a_{x_k}(s_k)\int_{0}^{s_k} \psi
(x_k u_t)dt - a_{x_k}((1-\e) s_k)\int_{0}^{(1-\e) s_k} \psi
(x_k u_t)dt .\end{multline*}

Since \begin{align*} & \left |a_{x_k}(
s_k)\int_{0}^{s_k} \psi (x_k u_t)dt - a_{x_k}((1-\e)
s_k)\int_{0}^{(1-\e) s_k} \psi (x_k u_t)dt \right | \\ &\le |a_{x_k}(s_k) +
a_{x_k}((1-\e)s_k)|\cdot \int_{0}^{s_k} \psi (x_k u_t)dt  \\
&\le \frac{|a_{x_k}(s_k)
+a_{x_k}((1-\e) s_k)|\int_{(1-\e)s_k }^{s_k} \psi (x_k u_t)dt }{c\cdot \e },
\end{align*} 
we obtain that
 $$
 \frac{ \int_{(1-\e)s_k}^{s_k} f (x_k
u_t)dt} {\int_{(1-\e)s_k}^{s_k} \psi (x_k u_t)dt}=
\frac{\mu_{x_k}(f)}{\mu_{x_k}(\psi)} + O\left(\tfrac{|a_{x_k}(s_k)
+a_{x_k}((1-\e) s_k)|}{c \e }\right).
$$
Since  $a_{x_k}(s_k)+a_{x_k}((1-\e) s_k)\to 0$, uniformly in $\{x_k\},$ 
the lemma follows.
\end{proof}

\section{Additional invariance and Ergodicity of BR for $\delta>1$}
Let $\G$ be a convex cocompact subgroup with $\delta>1$.


\subsection{Reduction} Let $\mathcal A, X'$ and 
$m^{\BR}=\int_{x\in X} \mu_x d (m^{\BR})_*(x)$
be the decomposition of $m^{\BR}$ into $U$-ergodic components, see Section \ref{erg}.

Our strategy in proving the $U$-ergodicity of
$m^{\BR}$ is to show that for a.e. $x\in X$, $\mu_x$ is $N$-invariant.

Fix a BMS box $E$ and a non-negative
function $\psi\in C_c(X)$ with $\psi|_E>0$.
Let $0<r<1$ be as in the window theorem \ref{wtpf} and $r_0:=\frac{r}{16}$.
Recall $E_{{r}_0}(\psi)\subset E$ from Lemma \ref{lem;unif-conv}. 

The next subsection is devoted to a proof of the following: \begin{thm}\label{fin}
 For any $x_0\in E_{r_0}(\psi)\cap \op{supp}(m^{\BR})$, $\mu_{x_0}$ is $N$-invariant.
\end{thm}

\begin{lem}\label{in10} There exists a $\op{BR}$-conull set $X''$ such that if $x, xn \in X''$ for $n\in N$, then
$\mu_{xn}=n. \mu_x$. 
\end{lem}
\begin{proof}
Since $N$ is abelian and $U<N$, $n.\mu_x$ is $U$-invariant 
and ergodic for every $n\in N$ and for a.e. $x$.
Now since $m^{\BR}$ is $N$-invariant, we have $m^{\BR}=\int n.\mu_x d (m^{\BR})_*(x)$ is
also a $U$-ergodic decomposition of $m^{\BR}$ for each $n\in N$.
The claim now follows from
 the uniqueness of ergodic decomposition.
\end{proof}

\begin{cor} \label{f4} $m^{\BR}$ is $U$-ergodic. 
\end{cor}
\begin{proof} Set $$F:=\{x\in X: \mu_x \text{ is $N$-invariant}\}.$$ By
 Lemma \ref{in10},  the characteristic  function $\chi_F$ is an $N$-invariant measurable function.  
 Since $m^{\BR}$ is $N$-ergodic by Theorem \ref{fm} and $m^{\BR}(F)>0$  by Theorem \ref{fin}, it follows that $m^{\BR}(X-F)=0$. 
That is, $\mu_x=m^{\BR}$ for a.e. $x$, and since $\mu_x$'s are $U$-ergodic components of $m^{\BR}$,
the claim follows.
\end{proof}

\subsection{Proof of Theorem \ref{fin}}
 As we explained in the introduction, we will flow two nearby points in the generic set and study their divergence in the ``intermediate range''. We first need to prove a refinement of the window theorem, see Propositions~\ref{pe} and~\ref{lem;swindow-y} below.
 
$$\text{ Fix $x_0\in E_{r_0}(\psi)\cap \text{supp}(m^{\BR})$.}$$
\begin{prop}
 There is a Borel subset $E'\subset
E$ such that $\mBR(E-E')=0$ and for any $x\in E'$ and all integers $m\ge 1$,
 $$xN \cap B(x_0, \tfrac 1m)\cap E_\rho(\psi)\ne \emptyset .$$
\end{prop}
\begin{proof} Set $N_k:=\{n_z: |z|<k\}$. Since $\mBR$ is $N$-ergodic, by \cite{Hochman},
there exists a full measure subset $E'_m$ of $E$ such that for  all $x\in E'_m$ 
$$\lim_k\frac{\int_{N_k}\chi_{B(x_0, 1/m)\cap E_{r_0}(\psi)}(xn_z) dz}{\int_{N_k}\psi(xn_z) dz}= 
\frac{m^{\BR}(B(x_0, 1/m)\cap E_{r_0}(\psi))}{m^{\BR}(\psi)} .$$ 
It suffices to take $E':=\cap_m E_m'$.
\end{proof}

Since $\inf_{x\in E_{r_0}(\psi)}\mu_x(\psi)>0$ and $x\mapsto
\frac{1}{\mu_x(\psi) } \mu_x$ is continuous on $E_{r_0}(\psi)$,
there exists a symmetric neighborhood $\mathcal O$  such that
\be\label{e;imp-cond-psi} 0<\inf_{g\in \mathcal O, x\in E_{r_0}(\psi) } \frac{|\mu_x(g\psi)|}{\mu_x(\psi)}\le
\sup_{g\in \mathcal O, x\in E_{r_0}(\psi) } \frac{|\mu_x(g\psi)|}{\mu_x(\psi)}<\infty .\ee 
Set $K_\psi:=\text{supp}(\psi)\mathcal O$
and $K_\psi':=\cap_{g\in \mathcal O}\text{supp}(\psi)g$. By Theorem \ref{wtpf},
for all $s\ge T_0$, the following set has BR measure at least $\frac{5r}{16}m^{\BR}(E)$:
  \begin{multline*}
E_s:=\{x\in E_{r_0}(\psi)\cap E_{{r_0}}(\chi_{K_\psi})\cap E_{r_0}(\chi_{K_\psi'}):\\ \int_{-rs}^{rs}\psi(xu_t) dt 
<(1-r)\int_{-s}^s\psi(xu_t) dt\}.\end{multline*}
Therefore for each $s\ge T_0$,
there exists a compact subset $\mathcal G(s)$ of $E_s\cap E'$  with $m^{\BR}( \mathcal G(s))> \frac {r}{8}
m^{\BR}(E) $.

We may write $\mathcal G(s)$ as $\mathcal G(s)_+\cup \mathcal G(s)_-$
where  $$\mathcal G(s)_+=\{x\in \mathcal G(s): 
\int_{0}^{rs}\psi(xu_t) dt 
\le \tfrac{1-r}{2}\int_{0}^s\psi(xu_t)dt\};$$
$$\mathcal G(s)_-=\{x\in \mathcal G(s): 
\int_{-rs}^{0}\psi(xu_t) dt 
\le \tfrac{1-r}{2}\int_{-s}^0\psi(xu_t)dt\}.$$
Therefore there exists an infinite sequence $p_i\to +\infty$ such that
$m^{\BR}(\mathcal G(p_i)_+)\ge \frac{r}{16}$ for all $i$
or $m^{\BR}(\mathcal G(p_i)_-)\ge \frac{r}{16}$ for all $i$.

In the following, we assume the former case that
$m^{\BR}(\mathcal G(p_i)_+)\ge \frac{r}{16}$ for all $i$. 
The argument is symmetric in the other case.

\begin{prop}\label{pe}\label{e;small-window} 
Fix integers $\ell,m >1$. There exist an infinite sequence $s_k=s_k(\ell, m)$ 
and elements
$x_k=x_k(\ell, m), y_k=y_k(\ell, m)\in \mathcal G({s_k})_+$ which satisfy the following:
\begin{enumerate}
 \item 
$y_k=x_k\chn_{w_k}$ where 
${c_1^{-1}}{s_k^{-2}}\ell^{-1} \leq |w_k|\leq
{c_1}{s_k^{-2}}\ell^{-1}$ and $|\Im(w_k)|\geq
\frac{|\Re(w_k)|}{c_1}$ where $c_1>1$ is independent of $\ell,\e, k$.
\item  each $x_k$  satisfies
\begin{equation*}\label{e;small-window11}\int_{(1-\e)s_k }^{s_k} \psi (x_k
u_t)dt\geq \frac{r}{4m}\int_{0} ^{s_k}\psi(x_ku_t)dt .\end{equation*}

\end{enumerate}
\end{prop}

\begin{proof}
If $x\in \mathcal G_+(s)$,  then, as $r<1$,
$$\int_{rs}^{s}\psi(xu_t)\ge  r\int_{0} ^s\psi(xu_t) dt.$$

By subdividing $[r,1]$ into $m$ subintervals
$I_i=(r+\tfrac{(j-1)}{m}, r+ \tfrac{j}{m} )$'s of length $\tfrac{1}{m}$,
there exists an integer  $1\le j=j(x,s)\le m$ such that
$$\int_{(r+{(j-1)}/{m})s }^{(r+{j}/{m} )s} \psi (x u_t)dt  \ge \frac{r }{4m} \int_{0} ^{s}\psi(xu_t)dt .$$

Let $d_0=d_0(r/16)>0$ be as in Proposition ~\ref{p;f-goodpts}. 
Applying Proposition~\ref{p;f-goodpts} to each $\mathcal G(p_i)_+$ 
and a sequence $(p_i\ell)^2$, we can find $x_{i}, y_i \in \mathcal G(p_i)_+$ satisfying 
$y_i=x_i\chn_{w_i}$ with $d_0^{-1} p_i^{-2}\ell^{-1}\le |w_i|\le d_0p_i^{-2}\ell^{-1}$
and $|\Im(w_i)|\ge \frac{|\Re{w_i}|}{d_0}$.
Choose a subsequence $x_{i_k}$ of $\{x_i\}$ such that
$j(x_{i_k}, p_{i_k})$ is a constant, say, $j_0$.
Setting $s_k:=(r+\tfrac{j_0}{m})p_{i_k}$, $x_k:=x_{i_k}$ and $y_k=y_{i_k}$, we have
$$
\int_{(1-\e)s_k }^{s_k} \psi (x_k
u_t)dt\geq \frac{r}{4m}\int_{0} ^{s_k}\psi(x_ku_t)dt$$
and $r p_{i_k} \le s_k \le (r+1) p_{i_k}$. Hence the claim follows with $c_1=d_0 (r+1)^2$.
\end{proof}

We now use the fact that
the two orbits $x_ku_t$ and $y_ku_t$ stay ``close'' to each other
for all $t\in[0,s_k],$ to show that $y_k$'s in Proposition \ref{pe} also satisfy the same type of window estimate. Let us fix some notation; writing $y_k u_t= x_k u_t  (u_{-t} {{\chn}_{w_k}} u_t)$,
we set $$p_k(t):=u_{-t} {{\chn}_{w_k}} u_t=\begin{pmatrix} 1 +tw_k & w_k \\
-t^2 w_k  & 1-tw_k
 \end{pmatrix} ,$$
and $g_k=p_k(s_k)$. 

\begin{prop}\label{lem;swindow-y} There are positive constants $c_2=c_2(\psi)$ and $\e_0=\e_0(\psi)$ such that for all $\e=\tfrac{1}{m}<\e_0$ and all $k\gg 1$,
\[\int_{(1-\e)s_k }^{s_k} \psi (y_k u_t)dt\geq c_2\cdot \e\int_{0} ^{s_k}\psi(y_ku_t)dt ,\]
where $y_k=y_k(\ell,\e)$ is as in Proposition \ref{pe}.
 \end{prop}

\begin{proof}
There is a constant $c>0$ (independent of
$\e$) such that $|p_k(t)g_k^{-1}|< c \e$ for all
$t\in[(1-\e)s_k,s_k]$. Hence for all $\ell\gg 1$ (independent of
$\e$), we have  $p_k(t)\in \mathcal O$ for all $t\in [0, s_k]$.

\noindent{\bf Claim (1)}:
For some constant $b_1>0$, independent of $\e,$ we have for all $k\gg 1$,
 \begin{equation}\label{pd}\int_0^{s_k}\psi(x_ku_t)dt\ge
b_1\int_0^{s_k}\psi(y_ku_t)dt .\end{equation}

By the definition of $K_\psi$ and $K_\psi'$,
since $y_ku_t\in x_ku_t\mathcal O$, we have
$\chi_{K_\psi'}(y_ku_t)\le\chi_{K_\psi}(x_ku_t)$ for all
$t\in[0,s_k].$ In particular we have
\[\int_0^{s_k}\chi_{K_\psi}(x_ku_t)dt\ge\int_0^{s_k}\chi_{K_\psi'}(y_ku_t)dt .\]
On the other hand, we have
\[\begin{array}{c}\int_0^{s_k}\psi(y_ku_t)dt=
\frac{\mu_{y_k}(\psi)}{\mu_{y_k}(\chi_{K_\psi'})}\int_0^{s_k}
\chi_{K_\psi'}(y_ku_t)dt+a_{x_k}(\psi,s_k) \int_0^{s_k}\chi_{K_\psi'}(x_ku_t)dt; \vspace{1mm}\\
\int_0^{s_k}\psi(x_ku_t)dt=\frac{\mu_{x_k}(\psi)}{\mu_{x_k}(\chi_{K_\psi})}\int_0^{s_k}\chi_{K_\psi}(x_ku_t)dt+
a_{x_k}(\psi, s_k)\int_0^{s_k}\chi_{K_\psi}(x_ku_t)dt\end{array}\]
with $\max\{|a_{x_k}(\psi,s_k)|, |a_{y_k}(\psi, s_k)|\}\le a(s_k)
\to 0$ as $k\to\infty$.

As $\frac{\mu_{y_k}(\psi)}{\mu_{y_k}(\chi_{K_\psi'})}$ and
$\frac{\mu_{x_k}(\psi)}{\mu_{x_k}(\chi_{K_\psi})}$ are uniformly bounded from below and above, by the choice of $x_k$ and
$y_k$, there exists $b>0$ such that for all large $k\gg 1$,
$$\int_0^{s_k}\psi(x_ku_t)dt\ge b
\int_0^{s_k}\chi_{K_\psi}(x_ku_t)dt\ge b
\int_0^{s_k}\chi_{K_\psi'}(y_ku_t)dt \ge b^2
\int_0^{s_k}\psi(y_ku_t)dt$$ finishing the proof of Claim (1).

\noindent{\bf Claim (2)}:  For some constant $b_2>0$, independent of $\e,$ we have for all $k\gg 1,$
  \begin{equation}\label{ppp}\int_{(1-\e)s_k}^{s_k}\chi_{K_\psi}(x_ku_t) dt \le
b_2 \int_{(1-\e)s_k}^{s_k}\psi(x_ku_t) dt .\end{equation}

By Lemma \ref{ew} and its proof,
 we have  \begin{multline}\label{e;unif-swindow1}\int_{(1-\e)s_k}^{
s_k} \chi_{K_\psi} (x_k
u_t)dt=\tfrac{\mu_{x_k}(\chi_{K_\psi})}{\mu_{x_k}(\psi)}\int_{(1-\e)s_k}^{
s_k} \psi (x_k u_t)dt\\ + \tfrac{4 (a(s_k)
+a(\e s_k))}{r\e } \cdot \int_{(1-\e)s_k }^{s_k} \psi (x_k u_t)dt.\end{multline}
 Since
$\frac{\mu_{x_k}(\psi)}{\mu_{x_k}(\chi_{K_\psi})}$ is
 uniformly bounded from above and below by positive constants,
 it suffices to take $k$ large enough so that
 $(a(s_k) +a(\e s_k))\le \e $ to finish the proof of Claim (2).

We have
\begin{align*}& \int_{(1-\e)s_k}^{s_k}\psi(y_ku_t)dt
=\int_{(1-\e)s_k}^{s_k}\psi(x_ku_tp_k(t))dt
\\ &\ge\int_{(1-\e)s_k}^{s_k}\psi(x_ku_t g_k)dt  -
\int_{(1-\e)s_k}^{s_k}|\psi(x_ku_tp_k(t))-\psi(x_ku_tg_k)|dt.\end{align*}
By \eqref{ppp}, for all large $k$,
\begin{align*}\int_{(1-\e)s_k}^{s_k}|\psi(x_ku_tp_k(t))-\psi(x_ku_tg_k)|dt &
\leq c_\psi \e\int_{(1-\e)s_k}^{s_k}\chi_{K_\psi}(x_ku_t) dt\\ &\le
c_\psi b_2 \e \int_{(1-\e)s_k}^{s_k}\psi(x_ku_t) \end{align*} 
where $c_\psi$ is the Lipschitz constant of $\psi$.
 Since $g_k\in \mathcal O$ and hence $\frac{\mu_{x_k}(g_k\psi)}{\mu_{x_k}(\psi)} $ is uniformly bounded from above and below,
we can deduce that for some $c>1$,
\begin{equation}\label{lem;unif-swindow} c^{-1}
\int_{(1-\e) s_k  }^{s_k} \psi (x_k u_t)dt\leq\int_{(1-\e) s_k }^{s_k} \psi (x_k u_tg_k)dt
\leq c\int_{(1-\e) s_k }^{s_k} \psi (x_k u_t)dt. \end{equation} 
Therefore the above estimates together with \eqref{pd} imply that for all $k$ large,
\begin{align*} & \int_{(1-\e)s_k}^{s_k}\psi(y_ku_t)dt\ge (c^{-1}-c_\psi b_2 \e)
\int_{(1-\e)s_k}^{s_k}\psi(x_ku_t)dt\\
&\ge \frac{ (c^{-1}-c_\psi b_2 \e)r\e}{4}
\int_{0}^{s_k}\psi(x_ku_t)dt\\ & \ge \frac{b_1 (c^{-1}-c_\psi b_2 \e)r\e}{4}
\int_{0}^{s_k}\psi(y_ku_t)dt.\end{align*}
Now the proposition follows with $c_2=\tfrac{b_1r}{8c}$ and $\e_0=\tfrac{1}{2b_2c_\psi c}$.
\end{proof} 

We will now flow $x_k$ and $y_k$ for the period of time $[(1-\e)s_k,s_k].$ By the construction of these points, these two pieces of orbits are almost parallel and they essentially differ by $g_k$ which is of size O(1). More importantly these ``short'' pieces of the orbits already become equidistributed. This will show that some ergodic component is invariant by a nontrivial element in $N-U$ and the proof can be concluded from there
 using standard arguments.  

Fix $\ell\in \n$. Let $\e_i=\frac{1}{i}>0$ for $i\in \n$. 
We choose $s_k(\e_1, \ell)$ and $x_k(\e_1,\ell),
y_k (\e_1,\ell)\in \mathcal G(s_k(\e_1, \ell))_+$ as in Proposition \ref{pe}.
Together with Proposition \ref{lem;swindow-y},
there exists  $\alpha_1>0$ independent of $\e_1$ and $k$ such that
\begin{equation*}\int_{(1-\e_1)s_k(\e_1,\ell) }^{s_k(\e_1,\ell)} \psi (x_k(\e_1,\ell)
u_t)dt\geq  \alpha_1 {\e_1}\int_{0} ^{s_k(\e_1,\ell)}\psi(x_k(\e_1,\ell)u_t)dt \quad \text{and}  \end{equation*} 
\begin{equation*} \label{e;small-window2}
\int_{(1-\e_1)s_k(\e_1,\ell) }^{s_k(\e_1,\ell)} \psi (y_k(\e_1,\ell) u_t)dt\geq \alpha_1  \e_1\int_{0} ^{s_k(\e_1,\ell)}\psi(y_k(\e_1,\ell)u_t)dt 
.\end{equation*}

 By passing to a
subsequence, we assume that $x_k (\e_1,\ell) \to x_{\e_1,\ell}$, and hence
$y_k (\e_1,\ell) \to y_{\e_1,\ell}$, and $p_k(s_k(\e_1,\ell))$
converges to $n_{v_{\e_1,\ell}}:=
\begin{pmatrix} 1  & 0
\\ v_{\e_1,\ell} & 1
\end{pmatrix}\in N$ where $\frac{1}{c_1\ell }\le |v_{\e_1,\ell}|\le \frac{c_1}{\ell}$
and $|\Im(v_{\e_1,\ell})|\geq\frac{|\Re(v_{\e_1,\ell})|}{c_1}$.

We proceed by induction: by dividing the interval $[(1-\e_i)s_k(\e_i, \ell),s_k(\e_i, \ell)]$ into
subintervals of length $\e_{i+1}$ as in the proof of Proposition \ref{pe},
we can find a sequence $s_k(\e_{i+1}, \ell)$ and subsequences
 $x_{k}(\e_{i+1}, \ell)$ of $x_{k}(\e_{i}, \ell)$ and
$y_{k}(\e_{i+1}, \ell)$ of $y_{k}(\e_{i}, \ell)$ 
satisfying
\begin{equation*}\label{sw1}\int_{(1-\e_{i+1})s_k(\e_{i+1},\ell) }^{s_k(\e_{i+1},\ell)} \psi (x_k(\e_{i+1},\ell)
u_t)dt\geq  \alpha_1 {\e_{i+1}}\int_{0} ^{s_k(\e_{i+1},\ell)}\psi(x_k(\e_{i+1},\ell)u_t)dt;  \end{equation*} 
\begin{equation*} \label{sw2}
\int_{(1-\e_{i+1})s_k(\e_{i+1},\ell) }^{s_k(\e_{i+1},\ell)} \psi (y_k(\e_{i+1},\ell) u_t)dt\geq 
\alpha_1  \e_{i+1}\int_{0} ^{s_k(\e_{i+1},\ell)}\psi(y_k(\e_{i+1},\ell)u_t)dt 
\end{equation*} 
and
$p_k(s_k(\e_{i+1},\ell))$ converges to some element $n_{v_{\e_{i+1},\ell}}:=
\begin{pmatrix} 1  & 0
\\ v_{\e_{i+1},\ell} & 1
\end{pmatrix}\in N$ where $\frac{1}{c_1\ell }\le |v_{\e_{i+1},\ell}|\le \frac{c_1}{\ell}$
and $|\Im(v_{\e_{i+1},\ell})|\geq\frac{|\Re(v_{\e_{i+1},\ell})|}{c_1}$.

Clearly,  as $i \to \infty$, we have $x_{k}(\e_i, \ell)\to x_{\e_1, \ell}$ and $y_{k}(\e_i,
\ell)\to x_{\e_1, \ell}$. By passing to a subsequence, we may assume
that $v_{\e_i,\ell}$ converges to an element $v_\ell\in N$. Note that $\frac{1}{c_1\ell }\le
|v_{\ell}|\le \frac{c_1}{\ell}$ and
$|\Im(v_\ell)|\geq\frac{|\Re(v_\ell)|}{c_1}$.

Let $\ell_0>1$ be large enough so that $n_{v_{\ell}}\in \mathcal O$ for all $\ell>\ell_0$.
\begin{prop}\label{qi} Let $\ell>\ell_0$ and set $x_\ell:= x_{\e_1, \ell}$.
 For any $f\in C_c(X)$, we have
$$ \frac{\mu_{x_\ell}(f)}{\mu_{x_\ell}(\psi)}=\frac{\mu_{x_\ell}(n_{v_{\ell}}.
f)}{\mu_{x_\ell}(n_{v_{\ell}} .\psi)} .$$
\end{prop}

\begin{proof} We claim that 
there exists a constant $b>0$
such that for each $i\ge 1$, the following holds for all $k\gg_i 1$:
\begin{equation}\label{finalcon}\left|\frac{\mu_{y_k(\e_i,\ell)} (f)}{\mu_{y_k(\e_i,\ell)}(\psi)}-\frac{\mu_{x_k(\e_i,\ell)}(n_{v_{\e_i,\ell}}. f)}
{\mu_{x_k(\e_i,\ell)}(n_{v_{\e_i,\ell}} .\psi)}\right|<b\e_i.\end{equation}
We first deduce the proposition from this claim. Since both
$y_k(\e_i,\ell), x_k(\e_i,\ell)$ belong to the set $E_{r/16}(\psi)$ and converge to $x_\ell$, and $f,\psi\in C_c(X)$
have compact supports,
$\mu_{y_k(\e_i,\ell)}(f)\to \mu_{x_\ell}(f)$ and
 $\mu_{y_k(\e_i,\ell)}(\psi)\to \mu_{x_\ell}(\psi)$ as $k\to \infty$. 

Since $n_{v_{\e_i,\ell}} .f $ converges to $ n_{v_{\ell}}.f $ pointwise  as $i\to \infty$ and
the supports of all functions involved
are contained in one fixed compact subset of $X$,
we have $\mu_{x_k(\e_i,\ell)}(n_{v_{\e_i,\ell}}. f)
\to \mu_{x_\ell}(n_{v_{\e_i, \ell}}. f)$
as $k\to \infty$. Similarly,
$\mu_{x_k(\e_i,\ell)}(n_{v_{\e_i,\ell}}. \psi)
 \to \mu_{x_\ell}(n_{v_{\e_i, \ell}}. \psi)$
as $k\to \infty$. Hence \eqref{finalcon}
implies, by taking $k\to \infty$, that
$$ \left| \frac{\mu_{x_\ell} (f)}{\mu_{x_\ell}(\psi)}-\frac{\mu_{x_\ell}(n_{v_{\e_i,\ell}}. f)}
{\mu_{x_\ell}(n_{v_{\e_i,\ell}} .\psi)}\right|\le b\e_i .$$
Now by taking $i\to \infty$, this proves the proposition as $\e_i\to 0$.

To prove Claim \eqref{finalcon}, fixing $\e:=\e_i$, we set $v=v_{\e_i}$, $s_k=s_k(\e_i)$, $x_k=x_k(\e_i,\ell)$ and
$y_k=y_k(\e_i,\ell)$ for simplicity.
By Lemma \ref{ew}, we have, as $k\to \infty$,
\be\label{e;hopf-swindow1}\frac{\int_{(1-\e)s_k}^{s_k} f
(y_k u_t)dt}{\int_{(1-\e)s_k}^{s_k} \psi (y_k
u_t)dt}\sim \frac{\mu_{y_k}(f)}{\mu_{y_k}(\psi)}.\ee

Since $n_{v_\ell}^+\in \mathcal O$,
similar calculation 
implies that, as $k\to \infty$,
\be\label{e;hopf-swindow2}\frac{\int_{(1-\e)s_k}^{s_k}
n_v\cdot f (x_k u_t)dt}{\int_{(1-\e)s_k}^{s_k} n_v.\psi
(x_k u_t)dt} \sim \frac{\mu_{x_k}(n_v.
f)}{\mu_{x_k}(n_v.\psi)}.\ee

Therefore the claim follows if we show for all large $k\gg_i 1$,
\be\label{ell}\left|\frac{\int_{(1-\e)s_k}^{s_k} f (y_k
u_t)dt}{\int_{(1-\e)s_k}^{s_k} \psi (y_k
u_t)dt}-\frac{\int_{(1-\e)s_k}^{s_k} n_v . f (x_k
u_t)dt}{\int_{(1-\e)s_k}^{s_k} n_v .\psi (x_k
u_t)dt}\right|\le b \e \ee for some $b>0$ independent of $\e$.
Let $c_f$ and $c_{\psi}$ denote the Lipschitz constants of
$f$ and $\psi$ respectively. Hence for all $t\in [(1-\e)s_k, s_k]$ and large $k\gg 1$, 
 $$|f(x_k u_t p_k(t))-f(x_k u_t v_\ell)|< c_f (\e +s_k^{-1}) \le 2\e
 c_f$$
and $$|\psi( x_ku_t p_k(t))-\psi(x_k u_t v_\ell)|< c_\psi (\e
+s_k^{-1})\le 2\e c_\psi .$$

We have
\begin{align*}
&\frac{\int_{(1-\e)s_k}^{s_k} f (y_k u_t)dt}{\int_{(1-\e)s_k}^{s_k} \psi (y_k u_t)dt}=
\frac{\int_{(1-\e)s_k}^{s_k} f (x_k u_tp_k(t))dt}{\int_{(1-\e)s_k}^{s_k} \psi (x_k u_tp_k(t))dt}
\\ & =\frac{\int_{(1-\e)s_k}^{s_k} n_v\cdot f (x_k u_t)dt}{\int_{(1-\e)s_k}^{s_k} 
\psi (x_k u_tp_k(t))dt}+\frac{\int_{(1-\e)s_k}^{s_k} f (x_k u_tp_k(t))-n_v\cdot
 f (x_k u_t)dt}{\int_{(1-\e)s_k}^{s_k} \psi (x_k u_tp_k(t))dt} .\end{align*}

Let $K_0=K_\psi\cup K_f$. Then the above estimate,~\eqref{lem;unif-swindow}
and~\eqref{e;hopf-swindow2} with $f=\chi_{K_{0}},$ imply that
\begin{align*}
\left|\frac{\int_{(1-\e)s_k}^{s_k} f (x_k u_tp_k(t))-
n_v .f (x_k u_t)dt}{\int_{(1-\e)s_k}^{s_k} \psi (x_k u_tp_k(t))dt}\right|& 
\le 2c_fc
\e\hh\frac{\int_{(1-\e)s_k}^{s_k}\chi_{K_{0}}(x_ku_t)dt}{\int_{(1-\e)s_k}^{s_k} \psi (x_k u_t)dt}\\ 
&\le 4c_fc \e\hh \frac{\mu_{x_k}(K_{0})}{\mu_{x_k}(\psi)}.\end{align*}
On the other hand we have
\begin{multline*} \frac{\int_{(1-\e)s_k}^{s_k} n_v. f (x_k u_t)dt}{\int_{(1-\e)s_k}^{s_k} \psi (x_k u_tp_k(t))dt}\\ =
\frac{\int_{(1-\e)s_k}^{s_k} n_v. f (x_k u_t)dt}{\int_{(1-\e)s_k}^{s_k} n_v.
\psi (x_k u_t)dt}\cdot  \left({1+\frac{\int_{(1-\e)s_k}^{s_k} \psi (x_k u_tp_k(t))-n_v.
\psi (x_k u_t)dt}{\int_{(1-\e)s_k}^{s_k} n_v\cdot \psi (x_k u_t)dt} }\right)^{-1}.\end{multline*}

Similar estimate as above gives
\[\left|\frac{\int_{(1-\e)s_k}^{s_k} \psi (x_k u_tp_k(t))-n_v .\psi (x_k u_t)dt}
{\int_{(1-\e)s_k}^{s_k} n_v.\psi (x_k u_t)dt}\right|\leq  4c_\psi c \e\hh \frac{\mu_{x_k}(K_{0})}{\mu_{x_k}(\psi)} .\]
All these together imply there exists a constant $c'>0$ (depending on $f$ and $\psi$ but independent of $\e$) such that
\begin{equation*}\label{e;y-to-x}\frac{\int_{(1-\e)s_k}^{s_k} f (y_k u_t)dt}{\int_{(1-\e)s_k}^{s_k} \psi (y_k u_t)dt}
=(1+c'\e)\frac{\int_{(1-\e)s_k}^{s_k} n_v . f (x_k u_t)dt}{\int_{(1-\e)s_k}^{s_k}
n_v .\psi (x_k u_t)dt}+c'\e .\end{equation*}

Now by ~\eqref{e;hopf-swindow1},~\eqref{e;hopf-swindow2}, this
implies the claim \eqref{ell}.

\end{proof}

The following proposition finishes the proof of Theorem \ref{fin}.
\begin{prop}
 $\mu_{x_0}$ is invariant under $N$. 
\end{prop}

\begin{proof} 
The set $\{n\in N: n.\mu_{x_0}=\mu_{x_0}\}$ is a closed subgroup which contains $U$.
Let $x_\ell$ and $v_\ell$ be as in Proposition \ref{qi}. Since $\frac{1}{c_1\ell }\le
|v_{\ell}|\le \frac{c_1}{\ell}$ and
$|\Im(v_\ell)|\geq\frac{|\Re(v_\ell)|}{c_1}$,
it suffices to show that $\mu_{x_0}$ is invariant under $n_{v_\ell}$ for all $\ell >\ell_0$.
Note that $x_{\ell}\in E_\rho(\psi)$. Set $N_0:=\{n\in N:
x_{\ell} n\in E_\rho(\psi)\}$. We have for any $n\in N_0$ and $f\in C_c(X)$,
$$\frac{\mu_{x_\ell}(n.f)}{\mu_{x_\ell}(n.\psi)}=\lim_T\frac{\int_0^Tf(x_\ell u_t n)dt}{\int_0^T\psi(x_\ell u_t n) dt}=
 \frac{\mu_{x_\ell n}(f)}{\mu_{x_\ell n}(\psi)}
.$$
On the other hand, by Proposition \ref{qi},
we have
$$\frac{\mu_{x_\ell}(n.f)}{\mu_{x_\ell}(n.\psi)}
= \frac{\mu_{x_\ell n_{v_\ell}}(n.f)}{\mu_{x_\ell. n_{v_\ell}}(n.\psi)}
= \frac{\mu_{x_\ell n}(n_{v_\ell}.f)}{\mu_{x_\ell. n}(n_{v_\ell}.\psi)}.$$

Therefore for any $n\in N_0$,
$$ \frac{\mu_{x_\ell n}(f)}{\mu_{x_\ell n}(n_{v_\ell}.f )}=\frac{\mu_{x_\ell n}(\psi)}{\mu_{x_\ell. n}(n_{v_\ell}.\psi)} (\ne 0).$$

 As $x_\ell\in E'$, it follows from the definition of $E'$ that
  we can take a sequence $n_m$ such that $x_\ell n_m \in E_\rho(\psi)\cap B(x_0, m^{-1})$
and hence $x_\ell n_m \to x_0$ as $m\to \infty$.

In particular, 
$$\frac{\mu_{x_0}(f)}{\mu_{x_0}(n_{v_\ell}.f )}=\lim_{m\to \infty} \frac{\mu_{x_\ell n_m}(f)}{\mu_{x_\ell n_m}
(n_{v_\ell}.f )}
=\lim_{m\to \infty} \frac{\mu_{x_\ell n_m}(\psi)}{\mu_{x_\ell n_m}(n_{v_\ell}.\psi )}
=\frac{\mu_{x_0}(\psi)}{\mu_{x_0}(n_{v_\ell}.\psi )}.$$

It follows that $\mu_{x_0}$ and $n_{v_{\ell}}.\mu_{x_0}$ are not mutually singular to each other. 
Hence by Lemma \ref{in10},
$\mu_{x_0} =n_{v_{\ell}}.\mu_{x_0}$.
\end{proof}

Finally we state the following: recall the notation $m_{N_0}^{\BR}$
 from the subsection \ref{gu}. 
\begin{thm}\label{guuuuu}
 If $U_0$ is a one-parameter unipotent subgroup of $G$ and $\G$ is a convex cocompact subgroup with $\delta>1$,
then $m_{N_0}^{\BR}$ is $U_0$-ergodic  for $N_0=C_G(U_0)$. 
\end{thm}
\begin{proof} Let $k_0\in K$ be such that $U_0=k_0^{-1}Uk_0$. If $B\subset X$ is a Borel subset which is $U_0$ invariant,
then $Bk_0$ is $U$-invariant. Hence by Corollary \ref{f4}, $m^{\BR}(Bk_0)=0$ or
$m^{\BR}(X-Bk_0)=0$. By the definition of  $m^{\BR}_{N_0}$, it follows that
$m^{\BR}_{N_0}(B)=0$ or
$m^{\BR}_{N_0}(X-B)=0$.
\end{proof}


\begin{thebibliography}{ZZZZZ}





\bibitem{Aa}
Jon Aaronson.
\newblock An introduction to Infinite Ergodic theory.
\newblock {\em AMS, Providence} 1997


\bibitem{Bab}
Martine Babillot.
\newblock On the mixing property for hyperbolic systems.
\newblock {\em Israel J. Math.}, 129:61--76, 2002.

\bibitem{Brin}
M. Brin.
\newblock Ergodic theory of frame flows.
\newblock {\em In Katok, A. (ed.), Ergodic theory and Dynamical systems.}, Vol 2.,
 Proceedings, Maryland, Boston, Birkh\"auser, 1982



\bibitem{Bowditch1993}
B.~H. Bowditch.
\newblock Geometrical finiteness for hyperbolic groups.
\newblock {\em J. Funct. Anal.}, 113 :245--317, 1993.

\bibitem{Burger1990}
M. Burger.
\newblock Horocycle flow on geometrically finite surfaces.
\newblock {\em Duke Math. J.}, 61, 779--803, 1990.

\bibitem{CT} R. Canary and E. Taylor.
\newblock Kleinian groups with small limit sets.
\newblock {\em Duke Math. J.}, 73, 371--381, 1994


\bibitem{EL} M. Einsiedler and E. Lindenstrauss.
\newblock Diagonal actions on Locally homogeneous spaces.
\newblock {\em In Homogeneous flows, moduli spaces and arithmetic}, 
Clay Math. Proc., 10, 155--241 (2010) 


\bibitem{F} K. Falconer.
\newblock Fractal Geometry.
\newblock {\em Wiley}, 2nd Ed.(2003)

\bibitem{FS}
L. Flaminio and R. Spatzier.
\newblock Geometrically finite groups, Patterson-Sullivan measures and Ratner's theorem.
\newblock {\em Inventiones}, 99,  601-626 (1990)

\bibitem{Hochman}
M. Hochman.
\newblock A Ratio ergodic theorem for multiparameter non-singular actions.
\newblock {\em Journal of EMS}, Vol 12, 365--383 (2010)

\bibitem{Hopf} E. Hopf.
\newblock Ergodentheorie.
\newblock{\em Number 5 serii Ergebnisse der Mathematik}, 1937, Springer-Verlag

\bibitem{HT} X.~Hu and J.~Taylor.
\newblock Fractal properties of products and projections of measures in $\bbr^n$.
\newblock{\em Math. Proc. Cambridge Philos. Soc.} {\bf 115} 1994, 527-544.



\bibitem{KS} A. Katok and R. Spatzier.
\newblock Invariant measures for higher rank hyperbolic abelian actions.
\newblock {\em Ergodic theory and Dynamical Systems.,} {\bf 16} 1996, 751--778


\bibitem{Kre}  U. Krengel.
\newblock Ergodic theorems.
\newblock {\em De Gruyter Studies in Mathematics}, vol 6, Walter de Gruyter, Berlin, NY 1985 

\bibitem{L}  E. Lindenstrauss.
\newblock Invariant measures and arithmetic quantum unique ergodicity.
\newblock {\em Annals of Math}, 163,165--219, (2006) 


\bibitem{LM}J.~L.~Lions and E. Magenes.
\newblock Non-Homogeneous Boundary Value problems and Applications, I.
\newblock{\em Springer-Verlag} (1972)

\bibitem{Mar1}
G.~A.~Margulis.
\newblock On the action of unipotent groups in the space of lattices. 
\newblock {\em In Gelfand, I.M. (ed.) Proc. of the summer school on group representations. Bolyai 
Janos Math. Soc., Budapest, 1971.} 365-370. Budapest: Akademiai Kiado (1975).


\bibitem{Mar2}
G.~A.~Margulis.
\newblock Indefinite quadratic forms and unipotent flows on homogeneous spaces.
\newblock {\em Proc. of ``Semester on dynamical systems and ergodic theory" (Warsa 1986)}
 399--409, Banach Center Publ., (23), PWN, Warsaw, (1989).



\bibitem{Mars}
J. M. Marstrand.
\newblock Some fundamental geometric properties of plane sets of fractional dimension.
\newblock {\em Proc. LMS} (1954), Vol 4, 257--302

\bibitem{Ma}
Pertti Mattila.
\newblock Geometry of sets and measures in Euclidean Spaces.
\newblock {\em Cambridge. Univ. Press} (1995)

\bibitem{Ma1}
P. Mattila.
\newblock Hausdorff dimension, Projections and the Fourier Transform.
\newblock {\em Publ. Math.} 3--48, (2004)

\bibitem{MO}
A. Mohammadi and H. Oh.
\newblock Matrix coefficients, Counting and Primes for orbits of geometrically finite groups 
\newblock{\em To appear in Journal of European Math. Soc..} arXiv:1208.4139


\bibitem{Mo}
C. Moore.
\newblock Ergodicity of Flows on Homogeneous Spaces.
\newblock {\em American J. Math.} (1966), Vol 88, 154--178

\bibitem{OS}
H. Oh and N. Shah.
\newblock Equidistribution and Counting for orbits of geometrically finite hyperbolic groups.
\newblock {\em Journal of AMS}, Vol 26 (2013), 511--562

\bibitem{OP}
J.P. Otal and M. Peigne.
\newblock Principe variationnel et groupes Kleiniens.
\newblock{\em Duke M. J}, 125, 15-44, 2004

\bibitem{Patterson1976}
S.J. Patterson.
\newblock The limit set of a {F}uchsian group.
\newblock {\em Acta Mathematica}, 136:241--273, 1976.

\bibitem{Peigne2003}
M. Peign{\'e}.
\newblock On the {P}atterson-{S}ullivan measure of some discrete group of
  isometries.
\newblock {\em Israel J. Math.}, 133:77--88, 2003.

\bibitem{PS}Y.~Peres and W.~Schlag.
\newblock {Smoothness of projections, Bernoulli convolutions and the dimension of exceptions.}
\newblock {\em Duke Math. J.} {\bf 102} no. 2 (2000), 193-251.


\bibitem{Ratner1}
M. Ratner.
\newblock On measure rigidity of unipotent subgroups of semi-simple groups.
\newblock {\em Acta Math.} (1990), 165, 229--309.


\bibitem{Ratner2}
M. Ratner.
\newblock On Raghunathan's measure conjecture.
\newblock {\em Annals of Math.} (1992), Vol 134, 545--607.




\bibitem{Roblin2003}
T. Roblin.
\newblock Ergodicit\'e et \'equidistribution en courbure n\'egative.
\newblock {\em M\'em. Soc. Math. Fr. (N.S.)}, (95):vi+96, 2003.

   \bibitem{Roblin2000} T. Roblin.
\newblock{Sur l'ergodicit\'e rationnelle et les propri\'et\'es
              ergodiques du flot g\'eod\'esique dans les vari\'et\'es
              hyperboliques.}
\newblock{\em Ergodic Theory Dynam. Systems}, (20), 1785--1819, 2000.

\bibitem{Rudolph1982}
D. Rudolph.
\newblock Ergodic behaviour of {S}ullivan's geometric measure on a
  geometrically finite hyperbolic manifold.
\newblock {\em Ergodic Theory Dynam. Systems}, 2:491--512 (1983),











\bibitem{Sullivan1979}
D. Sullivan.
\newblock The density at infinity of a discrete group of hyperbolic motions.
\newblock {\em Inst. Hautes \'Etudes Sci. Publ. Math.}, (50):171--202, 1979.


\bibitem{Sullivan1984}
D. Sullivan.
\newblock Entropy, {H}ausdorff measures old and new, and limit sets of
  geometrically finite {K}leinian groups.
\newblock {\em Acta Math.}, 153(3-4):259--277, 1984.


\bibitem{Wi} D. Winter.
\newblock Mixing of frame flow for rank one locally symmetric manifold and measure classification.
\newblock {\em Preprint}., arXiv:1403.2425.



\bibitem{Yau} S. T. Yau.
\newblock Harmonic functions on complete Riemannian manifolds
\newblock {\em Comm. Pure. Appl. Math.,}, 28: 201-228, 1975


\bibitem{Zw} R. Zweim\"uller.
\newblock Hopf's ratio ergodic theorem by inducing.
\newblock {\em Preprint.} 











\end{thebibliography}
\end{document}